\documentclass[leqno,10pt]{amsart}

\usepackage{amsmath}

\usepackage{amssymb}

\newtheorem{thm}{Theorem}[section]

\newtheorem{prop}[thm]{Proposition}
\newtheorem{lemma}[thm]{Lemma}

\theoremstyle{definition}
\newtheorem{defn}[thm]{Definition}
\newtheorem{ex}[thm]{Example}

\theoremstyle{remark}
\newtheorem{remark}[thm]{Remark}

\numberwithin{equation}{section}

\def\La{\boldsymbol\Lambda}
\def\a{\mathfrak{a}}
\def\b{\mathfrak{b}}
\def\x{\mathbf{x}}
\def\y{\mathbf{y}}

\def\C{\mathbb{C}}
\def\R{\mathbb{R}}
\def\RP{\mathbb{RP}}

\def\a{\mathfrak{a}}
\def\b{\mathfrak{b}}

\def\SO{\mathrm{SO}}
\def\GL{\mathrm{GL}}

\def\Q{\mathcal{Q}}

\def\P{\mathcal{P}}

\def\span{\mathrm{span\,}}

\begin{document}

\title[]{Surfaces in Laguerre Geometry}

\author{Emilio Musso}
\address{(E. Musso) Dipartimento di Scienze Matematiche, Politecnico di Torino,
Corso Duca degli Abruzzi 24, I-10129 Torino, Italy}
\email{emilio.musso@polito.it}

\author{Lorenzo Nicolodi}
\address{(L. Nicolodi) Di\-par\-ti\-men\-to di Scienze Ma\-te\-ma\-ti\-che, Fisiche e Informatiche,
Uni\-ver\-si\-t\`a di Parma, Parco Area delle Scienze 53/A,
I-43124 Parma, Italy}
\email{lorenzo.nicolodi@unipr.it}

\thanks{Authors partially supported by MIUR (Italy) under the PRIN project
\textit{Variet\`a reali e complesse: geometria, topologia e analisi armonica};
and by the GNSAGA of INDAM}





\begin{abstract}
This exposition
gives an introduction to the theory of surfaces in Laguerre geometry and surveys some results, mostly
obtained by the authors, about three important classes of surfaces in Laguerre geometry, namely $L$-isothermic, 
$L$-minimal, and generalized $L$-minimal surfaces. The quadric model of Lie sphere geometry is adopted for 
Laguerre geometry and the method of moving frames is used throughout. As an example, the Cartan--K\"ahler theorem
for exterior differential systems is applied to study the Cauchy problem for the Pfaffian differential system 
of $L$-minimal surfaces. This is an elaboration of the talks given by the authors at IMPAN, Warsaw, in September 2016. 
The objective was to illustrate, by the subject of Laguerre surface geometry, some of the topics presented
in a series of lectures held at IMPAN by G. R. Jensen on Lie sphere geometry and by B. McKay
on exterior differential systems.

\end{abstract}

\maketitle

\section{Introduction}\label{s:intro}
Laguerre geometry is a classical sphere geometry
that has its origins in the work of E. Laguerre in the mid 19th century and that had been extensively
studied in the 1920s by Blaschke and Thomsen \cite{Blaschke1, Blaschke}.
The study of surfaces in Laguerre geometry is currently still an active area of research
\cite{AGM, li-li-wang, li-wang-mm, MN-REND-RM, MN-TAMS, MN-BOLL, MN-IJM, MN-TOH, MN-MZ, Pa1999, Pa1990} and
several classical topics in Laguerre geometry,
such as
Laguerre minimal surfaces and Laguerre isothermic surfaces and their transformation theory,
have recently received much
attention in the theory of integrable systems \cite{MN-AMB, MN-IJM, Pember, Szer, Rog-Szer},
in discrete differential geometry, and in the applications
to geometric computing and architectural geometry
\cite{Bobenko2006, Bobenko2007, Bobenko2010, Pott1998, Pott2009, Pott2012}.

\vskip0.2cm
The fundamental elements of Laguerre geometry in Euclidean space $\mathbb R^3$
are oriented planes, oriented spheres and points.
The orientation is determined by a field of unit normals
and the basic relation among these elements is that of oriented contact.
Laguerre geometry distinguishes between oriented planes and oriented spheres (including point spheres).
To describe oriented planes, oriented spheres and points in a uniform
way we will use
the quadric model of Lie sphere geometry \cite{Blaschke, Ce, Gary, JMN}.
In this model,
all oriented planes, oriented spheres and point spheres in $\mathbb R^3 \cup\{\infty\}$
are in bijective correspondence with the points on a quadric $\mathcal Q$ in real
projective space $\mathbb{RP}^5$
given by the equation $\langle \a, \a\rangle =0$, where $\langle\,,\rangle$ is
a nondegenerate scalar product of signature $(4,2)$ in $\mathbb{R}^{6}$.
Oriented contact is determined by the
conjugacy relation with respect to $\mathcal Q$.
The 5-dimensional space $\La$ of lines in $\mathcal Q$ which do not meet the
``point at infinity'' $\{\infty\}$ is the
underlying space for Laguerre geometry.
The transformations of Laguerre geometry
are bijective, separately, on the set $\mathcal Q_\Pi$ of oriented planes
and the set $\mathcal Q_\Sigma$ of oriented spheres (including point spheres) in $\mathbb{R}^3$,
fix the ``point at infinity'' and preserve oriented contact.
In the quadric model, the group of Laguerre transformations is a subgroup of the group of Lie sphere
transformations, which leave invariant the space of all oriented spheres, oriented planes
and point spheres in $\mathbb R^3 \cup \{\infty\}$.
The Laguerre group is isomorphic to the 10-dimensional (restricted) Poincar\'e group
\cite{Blaschke, Ce}.

\vskip0.2cm
An immersed surface $x : S \to \mathbb{R}^3$, oriented by a unit normal
field $n:S \to S^{2}$, induces a Legendre immersion
$f=(x,n):S\to \mathbb{R}^3\times S^{2} = T_1\mathbb{R}^3$ into the
unit tangent bundle of $\mathbb{R}^3$, endowed
with the natural contact structure given by the 1-form $\alpha_{(x,n)}=dx\cdot n$.
The space $\mathbb{R}^3\times S^{2}$
can be identified with $\La$ as contact manifold.
%
The {\it Laguerre space} is $\La\cong \mathbb{R}^3\times S^{2}$ viewed as homogeneous space of the Laguerre group.
The Laguerre group acts on the Legendre lifts of the immersions rather than on the
immersions themselves, since it does not act by point-transformations.
The principal aim of Laguerre
geometry is to study the properties of an immersion which are invariant
under the action of the Laguerre group on Legendre immersions.
Locally and up to Laguerre transformation, any
Legendre immersion arises as a Legendre lift.
Thus, the study of
Laguerre geometry of surfaces in $\mathbb{R}^3$ is reduced
to that of Legendre immersions in $\La \cong \mathbb{R}^3\times S^{2}$.
The most important tool to study the Laguerre geometry of surfaces
in $\mathbb{R}^3$
is the Laguerre position vector $\mathfrak s : S \to \mathbb{R}^{6}$,
$\langle \mathfrak s, \mathfrak s\rangle =0$,
which corresponds to the classical {\it middle sphere
congruence},
also known as the {\it Laguerre Gauss map}.
For a given immersion $x : S \to \mathbb{R}^3$, by
pulling back $\langle\,,\,\rangle$ via $\mathfrak s$
we obtain a Laguerre invariant metric on $S$. This metric
is conformal to the third fundamental form $\mathrm{III}= dn\cdot dn$ of the immersion $x$.

\vskip0.2cm

The purpose of this article is to survey some significant results concerning three important
classes of surfaces in Laguerre geometry,
namely, $L$-isothermic, $L$-minimal, and generalized $L$-minimal surfaces.
These are the Laguerre geometric counterparts
of isothermic, Willmore, and generalized Willmore surfaces in M\"obius
geometry (cf. Remark \ref{r:moebius-ge}).
%
Traditionally, see the work of
Blaschke and Thomsen \cite{Blaschke}, Laguerre and M\"obius surface geometries have been
developed in parallel as subgeometries of Lie sphere geometry.

\vskip0.2cm

In Section \ref{s:pre}, we recall the basic facts about Laguerre geometry in Euclidean space
using the quadric model of Lie sphere geometry.

\vskip0.2cm
In Section \ref{s:nondeg-Leg-imm},
we develop the method of moving frames for Legendre immersions in the
Laguerre space $\La$
and show how to construct a canonical frame field for a Legendre immersion
under the hypothesis of nondegeneracy.\footnote{cf. Section \ref{s:nondeg-Leg-imm},
Definition \ref{nondegenerate}.}

\vskip0.2cm
In Section \ref{s:L-isotermic}, we discuss $L$-isothermic surfaces.
A surface in $\R^3$
is {\it $L$-isothermic} if, away from umbilic and parabolic points, it
admits curvature line coordinates that are isothermal (conformal)
with respect to the third fundamental form \cite{Blaschke, MN-BOLL, MN-AMB}.
$L$-isothermic surfaces are invariant under the Laguerre group and
are characterized as the only $L$-applicable
surfaces.\footnote{cf. Section \ref{s:L-isotermic},
Definition \ref{d:applicable}.}
Moreover, the nontrivial family of surfaces that are $L$-applicable
on a given $L$-isothermic surface depends on one parameter. This feature
indicates that $L$-isothermic surfaces constitute an integrable system.
For more details on the aspects of $L$-isothermic surfaces related to
the theory of integrable system and for the study of their rich
transformation theory,
including the analogues of the $T$-transformation and of the Darboux transformation
in M\"obius geometry, we refer the reader to
\cite{MN-BUD, MN-BOLL, MN-AMB, MN-IJM, MN-TOH, MN-MZ, Szer, Rog-Szer}.

\vskip0.2cm
In Section \ref{s:L-minimal}, we give an account of the theory of $L$-minimal surfaces.
A smooth immersed surface in $\R^3$ with no parabolic points is called
{\it Laguerre minimal} ($L$-{\it minimal}) if it is an extremal of the
Weingarten functional $$\int(H^2/K- 1)dA,$$ where $H$ and $K$ are the mean and
Gauss curvatures of the immersion, and $dA$ is the induced area element
of the surface \cite{Blaschke1, Blaschke, MN-TAMS, Pa1999}. The functional
and so its critical points are preserved by the Laguerre group.
For a Legendre immersion $f=(x,n) : S\to \La$,
the {\it Laguerre Gauss map}
is the map $\sigma_f : S \to \mathcal Q_\Sigma$
%
which assigns to each $s\in S$
the {\it middle sphere}, that is, the
oriented sphere of $\R^3$
of radius $H/K$ which is in oriented contact with the tangent plane
of $x$ at $x(s)$.
Away from umbilics and parabolic points, it turns out that $\sigma_f$ is a
spacelike immersion
into $\mathcal Q_\Sigma$, which is isometric to Minkowski 4-space $\R^{3,1}$.
Actually, $\sigma_f$ is a marginally outer trapped surface (MOTS) in $\R^{3,1}$ (cf. \cite{MNjgp, MN-MZ}
and the literature therein).
%
Interestingly enough,
one can prove that $\sigma_f$ has zero mean curvature vector
in $\mathcal Q_\Sigma \cong \R^{3,1}$ if and only if $f$ is $L$-minimal \cite{Blaschke, MN-TAMS}.
Associated to an arbitrary
nondegenerate
Legendre surface there is a naturally defined quartic differential $\Q_f$
and a quadratic differential $\P_f$ (cf. Section \ref{s:gen-L-min}).
An interesting feature of $L$-minimal immersions is that
this quartic differential is holomorphic \cite{MN-TAMS}.

\vskip0.2cm
In Section \ref{s:Pfaffian-sys}, $L$-minimal surfaces are interpreted as integral manifolds of a
suitable quasi-linear Pfaffian differential system on the manifold $M = L \times \mathbb R^4$.
The basic techniques of the theory of exterior differential systems are applied to study
the Cauchy problem for this system. We first use Cartan's test to prove that the system is in involution
and its solutions depend on four functions in one variable. Then, we compute the
polar equations of the system and apply the Cartan--K\"ahler theorem to solve the Cauchy problem.

\vskip0.2cm
In Section \ref{s:gen-L-min}, nondegenerate Legendre immersions $f$ whose associate
quartic differential $\Q_f$ is holomorphic are considered. Such immersions are called {\it generalized $L$-minimal}.
%
We prove that
a nondegenerate Legendre immersion
$f :S \to  \La$ is generalized $L$-minimal
if and only if the immersion is $L$-minimal
(in which case $\P_f$ is zero),
or is locally the $T$-transform of an $L$-minimal isothermic surface.
Using this, $L$-minimal isothermic surfaces and their $T$-transforms
can be characterized in terms of the differential geometry of their Laguerre Gauss maps.
It is proved that a Legendre immersion $f : S \to  \La$ is $L$-minimal and $L$-isothermic
if and only if
its Laguerre Gauss map $\sigma_f : S \to \R^{3,1}$
has zero mean curvature in some
spacelike, timelike, or (degenerate) isotropic hyperplane of $\R^{3,1}$.
Moreover, $f$ is generalized $L$-minimal
with non-zero $\P_f$
if and only if
its Laguerre Gauss map $\sigma_f : S \to \R^{3,1}$
has constant mean curvature $H=r$
in some translate of hyperbolic 3-space $\mathbb H^3(-r^2)\subset \R^{3,1}$,
de Sitter 3-space $\mathbb S^3_1(r^2)\subset \R^{3,1}$,
or has zero mean curvature in some translate of a time-oriented lightcone
$\mathcal L^3_\pm\subset \R^{3,1}$.
If the Laguerre Gauss map of $f$
takes values in a spacelike (respectively, timelike, isotropic) hyperplane, then the
Laguerre Gauss maps of the $T$-transforms of $f$ take values
in a translate of a hyperbolic 3-space (respectively, de Sitter 3-space,
time-oriented lightcone).
As an application of these results (cf. \cite{MN-MZ}), various instances of the Lawson
isometric correspondence \cite{Lawson} between
certain isometric constant mean curvature (CMC) surfaces in different hyperbolic 3-spaces
and of
the generalizations of Lawson's correspondence
in the Lorentzian \cite{Pa1990} and the (degenerate) isotropic situations
\cite{AiAk, AGM, Kob, Lee2005},
can be viewed as special cases of the $T$-transformation
of $L$-isothermic surfaces with holomorphic quartic differential.

\vskip0.2cm
For the subject of exterior differential systems and the Cartan--K\"ahler theorem we refer the reader to the monographs
\cite{BCGGG, Cartan-book, Grbook, GJbook, ILlibro} and \cite{McKay}. The summation convention over repeated indices is used
throughout the paper.

\section{Preliminaries and definitions}\label{s:pre}


\subsection{Notation}\label{ss:notation}

Vectors in $\R^6$ are denoted by $\mathfrak a, \mathfrak b, \dots$, and the components of a
vector $\mathfrak a$ with respect to the standard basis
${\boldsymbol\epsilon}_0,\dots,\boldsymbol\epsilon_5$ are indicated by $\mathfrak a^j$, $j =0,\dots,5$.
Vectors in $\R^4$ are denoted by capital letters $V, W, \dots$, and the components of a
vector $V$ with respect to the standard basis
$e_0,\dots,e_3$ are indicated by $V^a$, $a =0,\dots,3$. Vectors of $\R^3$ are denoted by
$u,v, \dots$, and the components of a
vector $v$ with respect to the standard basis are indicated by $v^i$, $i =1,2,3$.

For $X\in\mathfrak{gl}(3,\R)$ and $B \in \mathfrak{gl}(4,\R)$, we let
\[
 \widehat X = \begin{bmatrix} 1 & 0\\
 0 & X\end{bmatrix} \in \mathfrak{gl}(4,\R), \quad \widetilde B = K B K^{-1} \in \mathfrak{gl}(4,\R),
\]
where
\[
 K =  \begin{bmatrix}
 \frac{1}{\sqrt2} & \frac{1}{\sqrt2} & 0 & 0\\
 0 & 0 & 1 & 0\\
 0 & 0 & 0 & 1 \\
 \frac{1}{\sqrt2} & -\frac{1}{\sqrt2} & 0 & 0\end{bmatrix},
 \quad
 K^{-1} =  \begin{bmatrix}
 \frac{1}{\sqrt2} & 0 & 0 & \frac{1}{\sqrt2}\\
 \frac{1}{\sqrt2} & 0 & 0 & -\frac{1}{\sqrt2}\\
 0 & 1 & 0 & 0 \\
0 & 0 & 1 & 0\end{bmatrix}.
  \]

  For $v \in \R^3$ and $V \in \R^4$, we let $\widehat v$, $\widetilde V$ and $\mathfrak a(V)$ be given by
\[
 \widehat v = {}^t\!(0,v) \in \R^4, \quad \widetilde V = K V, \quad \mathfrak a(V) = {}^t\!\Big(1, \widetilde V, \frac{(V,V)}{2}\Big),
\]
where $(\,,\,)$ denotes the Lorentz scalar product of $\R^4$ given by
\[
 (V,W) = -V^0W^0 + V^1W^1 + V^2 W^2 + V^3W^3.
  \]
  We let $\R^{3,1}$ denote $\R^4$ with the above Lorentz scalar product.

\subsection{The Laguerre group}

Let $\R^{4,2}$ denote $\R^6$ with the orientation induced by the standard basis
${\boldsymbol\epsilon}_0,\dots,\boldsymbol\epsilon_5$ and
the scalar
product of signature $(4,2)$ given by
\begin{equation}\label{lie-sp}
\langle \a,\b\rangle=-(\a^0\b^5 + \a^5\b^0)-(\a^1\b^4 + \a^4\b^1)  +\a^2\b^2 +\a^3\b^3 = {}^t\!\a\eta\b,
\end{equation}
where
\[
\eta = (\eta_{ij}) = \begin{bmatrix} 0 & 0 & -L\\
0& I_2 & 0\\
-L & 0 & 0\end{bmatrix}, \quad L =  \begin{bmatrix} 0& 1\\1 & 0\end{bmatrix},
\quad I_2=  \begin{bmatrix} 1& 0\\0 & 1\end{bmatrix}.
\]

Let $G$ be the identity component of the pseudo-orthogonal group $\mathrm O(\eta)$ of \eqref{lie-sp},
\[
   \mathrm O(\eta)=\{A \in \GL(6,\R) \, : \, {}^t\!A\eta A = \eta\}.
     \]
The {\it Laguerre group} $L$ is the subgroup of $G$ consisting of all $A$
for which $A(\boldsymbol\epsilon_5) = \boldsymbol\epsilon_5$,
i.e.,
\[
 L := \{A \in G \, : \, A(\boldsymbol\epsilon_5) = \boldsymbol\epsilon_5\}.
  \]
In other words, $L$ is the 10-dimensional connected Lie group consisting of all
unimodular $6\times 6$ matrices $A$ satisfying the conditions
\[
 \langle A\x,A\y\rangle=\langle \x,\y\rangle,
  \quad -\langle A(\boldsymbol\epsilon_1 + \boldsymbol\epsilon_4),\boldsymbol\epsilon_1 + \boldsymbol\epsilon_4\rangle\geq 2,
   \quad A(\boldsymbol\epsilon_5) = \boldsymbol\epsilon_5.
    \]
The Lie algebra of $L$ is
\[
 \mathfrak l = \{ a = (  a^j_i) \in \mathfrak{gl}(6,\R) \, : \, \eta_{ik}  a^k_j
   + \eta_{jk}  a^k_i =0; \,   a^j_5 = 0 \}.
    \]

\subsubsection{Lorentz transformations, rotations, and boosts}\label{lor-rot-boo}

Let $G_6$ denote the group of {\it restricted Lorentz transformations} of $\R^{3,1}$ and
$P_{10} = G_6 \rtimes \R^4$ the corresponding inhomogeneous group, the {\it restricted Poincar\'e group}.
Retaining the notation of \S \ref{ss:notation}, every $A\in L$ can be written in the form
\begin{equation}\label{matrix-repr-lag-el}
 A(B;V) =
 \begin{bmatrix}
 1 & 0 & 0\\
 \widetilde V & \widetilde B& 0 \\
 \frac{(V,V)}{2} & {}^\ast\!\widetilde V \widetilde B & 1 \\
  \end{bmatrix},
\end{equation}
where $V \in\R^4$, $B\in G_6$, and
for each $W= {}^t\!(W^0,W^1,W^2,W^3)\in \R^4$,
\[
  {}^\ast\!\widetilde V W = -(\widetilde V^0 W^3 + \widetilde V^3 W^0) + \widetilde V^1 W^1
  +\widetilde V^2 W^2.
   \]

The mapping
\begin{equation}\label{iso-groups}
   G_6 \rtimes \R^4 \ni (B; V) \longmapsto A(B;V) \in L
   \end{equation}
is an isomorphism of Lie groups, which also gives a representation of the group $\mathbb E(3)= \SO(3) \rtimes \R^3$
 of Euclidean motions into the Laguerre group,
 \begin{equation}\label{euc-group-incl}
  \SO(3) \rtimes \R^3 \ni (R; v) \longmapsto A(\widehat R;\widehat v) \in L.
   \end{equation}
   The image of \eqref{euc-group-incl} coincides with the closed subgroup of $L$ consisting of all elements $A$
   fixing the vector $\boldsymbol\epsilon_1 + \boldsymbol\epsilon_4$. More explicitly, for
   $(R;v)\in \SO(3) \rtimes \R^3$, where $R=(r^i_j)$ and $v={}^t\!(v^1,v^2,v^3)$, we compute
\begin{equation}\label{exp-euc-group-incl}
 A(\widehat R;\widehat v)=
 \begin{bmatrix}
  1                 &      0                  &     0      &       0     &    0        &  0  \\
 \frac{v^1}{\sqrt2}  & \frac{1+r^1_1}{2}       & \frac{r^1_2}{\sqrt2} & \frac{r^1_3}{\sqrt2}  &  \frac{1-r^1_1}{2} &  0   \\
   v^2            & \frac{r^2_1}{\sqrt2}    &     r^2_2  & r^2_3    & \frac{-r^2_1}{\sqrt2}  &  0   \\
     v^3            & \frac{r^3_1}{\sqrt2}    &    r^3_2   & r^3_3    & \frac{-r^3_1}{\sqrt2}  &  0   \\
  \frac{-v^1}{\sqrt2} & \frac{1-r^1_1}{2}       & \frac{-r^1_2}{\sqrt2} & \frac{-r^1_3}{\sqrt2} & \frac{1+r^1_1}{2} &  0   \\
\frac{{}^t\!vv}{2}  & \frac{r^i_1v^i}{\sqrt2} & r^i_2v^i    &  r^i_3v^i & \frac{-r^i_1v^i}{\sqrt2}  &  1   \\
 \end{bmatrix}.
\end{equation}

Let $B_b\in G_6$ be the {\em pure Lorentz transformation} ({\it pure Lorentz boost}) of $\R^{3,1}$ given by
the symmetric matrix
   \[
 B_b =
 \begin{bmatrix}
 \beta & -\beta\, {}^t\!b\\
  -\beta\,  {b} & I_3 + \frac{\beta- 1}{{}^t\!b b} b {}^t\!b\\
  \end{bmatrix},
   \]
where $b \in \R^3$, ${}^t\!b b <1$, is the relative-velocity vector
and $\beta = (1 -{}^t\!b b)^{-\frac12}$.
We let $A(b) \in L$ denote the element of the Laguerre group corresponding to $B_b$ by means of \eqref{iso-groups},
i.e., $A(b) = A(B_b;0)$.
Since every restricted Lorentz transformation $B\in G_6$ has a unique decomposition as a product of a pure rotation
$\widehat R$, $R\in \SO(3)$, followed by a pure Lorentz boost $B_b$, $B = B_b \widehat R$,
it follows that every Laguerre transformation $A\in L$
can be decomposed
as a product
\[
 A = E A(b) T(s),
  \]
where, via \eqref{iso-groups} and \eqref{euc-group-incl}, $E$ corresponds to a Euclidean motion,
$A(b)$ corresponds to a pure Lorentz transformation
with velocity vector $b\in \R^3$, and $T(s)$ corresponds to the time translation
$\tau_s: {}^t\!(V^0,V^1,V^2,V^3) \mapsto {}^t\!(V^0+s,V^1,V^2,V^3)$ of the Poincar\'e group.

\subsection{The structure equations of Laguerre group}
For every $A \in L$, let $A_j = A \boldsymbol\epsilon_j$ denote the $j$th column vector of $A$. Thus,
$(A_0,\dots, A_5)$ is a {\it Laguerre frame}, i.e., a basis of $\R^6$ such that
\begin{equation}\label{lframe1}
 \langle A_i,A_j\rangle=\eta_{ij}, \quad A_5=\boldsymbol\epsilon_5
    \end{equation}
and $-\langle A_1 + A_4,\boldsymbol\epsilon_1 + \boldsymbol\epsilon_4\rangle\geq 2$, $i,j = 0,\dots, 5$.
Regarding the columns $A_j$, $j = 0,\dots,5$, of $A$ as
$\R^6$-valued functions,
there are unique 1-forms
$\{\omega^i_j\}_{\{i,j =0,\dots,5\}}$ so that
\begin{eqnarray}
  dA_j&=&\omega^i_jA_i,\label{lframe2a}
   \end{eqnarray}
where $\omega = (\omega^i_j) = A^{-1}dA$  is the the Maurer--Cartan form of $L$,
 i.e., the left-invariant $\mathfrak l$-valued 1-form $\omega = A^{-1}dA$.

Exterior differentiation of \eqref{lframe1} and \eqref{lframe2a} yields
the {\it structure equations of the Laguerre group}:
\begin{eqnarray}
 0 &=& \omega^k_i \eta_{kj}+\omega^k_j \eta_{ki}, \label{str-eq1}\\
 \omega^k_5 &=& 0,\label{str-eq3}\\
 d\omega^i_j &=& -  \omega^i_k \wedge \omega^k_j.\label{str-eq2}
\end{eqnarray}
It follows that
 \[
 \omega = (\omega^i_j) =
  \begin{bmatrix}
  0          &      0     &     0      &       0     &    0        &  0  \\
  \omega^1_0 & \omega^1_1 & \omega^1_2 & \omega^1_3  &    0        &  0   \\
  \omega^2_0 & \omega^2_1 &     0      & -\omega^3_2 & \omega^1_2  &  0   \\
  \omega^3_0 & \omega^3_1 & \omega^3_2 &       0     & \omega^1_3  &  0   \\
  \omega^4_0 &    0       & \omega^2_1 &  \omega^3_1 & -\omega^1_1 &  0   \\
     0       &-\omega^4_0 & \omega^2_0 &  \omega^3_0 & -\omega^1_0 &  0   \\
 \end{bmatrix} .
 \]

The left-invariant $\mathfrak l$-valued 1-form $\omega$ transforms by right translations according to
the rule
\[
R^\ast_A (\omega) = A^{-1}\omega A, \quad \forall \,\, A \in L.
\]

\subsection{Laguerre geometry in Euclidean 3-space: The quadric model}

A null line in $\R^{4,2}$ is a 1-dimensional subspace $[\mathfrak a] \subset \R^{4,2}$
spanned by a vector $\mathfrak a$ such that $\langle \mathfrak a, \mathfrak a \rangle = 0$.
The space of all null lines gives rise to the quadric $\mathcal Q$ in real projective space $\RP^5$
defined by the equation
\[
  \langle \mathfrak a, \mathfrak a \rangle =-2\a^0\a^5 -2\a^1\a^4 + (\a^2)^2   +(\a^3)^2 = 0.
  \]
  The quadric $\mathcal Q$ is called the {\em Lie quadric} \cite{Ce, J}.

The Laguerre group $L$ acts on the left on $\mathcal Q$ by $A\cdot [\mathfrak a] = [A\mathfrak a]$.
This action has a fixed point $\{[\boldsymbol\epsilon_5]\}$ (the ``point at infinity'') and two
nontrivial orbits,
\[
  \mathcal Q_\Sigma = \left\{ [\mathfrak a] \in \mathcal Q \, : \, \langle \mathfrak a,\boldsymbol\epsilon_5 \rangle
  \neq 0  \right\},
   \]
\[
\mathcal Q_\Pi = \left\{ [\mathfrak b] \in \mathcal Q \, : \, \langle \mathfrak b,\boldsymbol\epsilon_5 \rangle
  = 0, \,  \b \wedge \boldsymbol\epsilon_5 \neq 0 \right\}.
  \]
$\mathcal Q_\Sigma$ is a principal orbit, which is open and dense, while $\mathcal Q_\Pi$ has dimension 3.

\vskip0.2cm
We will now provide a geometric description of the orbits $\mathcal Q_\Sigma$ and $\mathcal Q_\Pi$.

\subsubsection{The space of oriented spheres}\label{ss:or-spheres}

The orientation of a 2-sphere in Euclidean space $\R^3$ is determined
by specifying a field of unit normals: inward normals if the radius is positive, outward normals if the radius
is negative. An oriented sphere with center $p$ and signed radius $r$ will be denoted by $S_r(p)$. We allow $r=0$,
in which case the sphere represents the point $p$; point spheres are not oriented. An oriented sphere $S_r(p)$,
$p = {}^t\!(p^1,p^2,p^3)$,
can be represented,
alternatively, by a vector $V(r,p) = {}^t\!(r,p^1,p^2,p^3)$ of $\R^4$ or by a null line $[\mathfrak a]$, spanned
by
\begin{equation}\label{a-r-p}
 \mathfrak a(r,p) = {}^t\!\Big(1, \frac{r +p^1}{\sqrt2},p^2,p^3, \frac{r -p^1}{\sqrt2},  \frac{{}^t\!pp -r^2}{2}\Big).
  \end{equation}
These two representations of oriented spheres are equivalent in that they are related by the mapping
\begin{equation}\label{equi-diff}
\R^4 \ni V \longmapsto [\mathfrak a(V)] \in \mathcal Q_\Sigma,
  \end{equation}
which is an equivariant diffeomorphism with respect of the action of the (restricted) Poincar\'e group on $\R^4$ and
of the Laguerre Group $L$ on $\mathcal Q_\Sigma$, taking into account the Lie group isomorphism defined by \eqref{iso-groups}.
The orbit $\mathcal Q_\Sigma$ may thus be identified with Minkowski 4-space $\R^{3,1}$ acted on by the (restricted) Poincar\`e group.

In particular, point spheres are represented by the elements of
\[
 \mathcal Q_\Sigma' = \left\{ [\mathfrak a] \in \mathcal Q_\Sigma \, : \, \langle \mathfrak a,\boldsymbol\epsilon_1
 + \boldsymbol\epsilon_4\rangle
  =0  \right\},
\]
and proper oriented spheres by the elements of
\[
 \mathcal Q_\Sigma'' = \left\{ [\mathfrak a] \in \mathcal Q_\Sigma \, : \, \langle \mathfrak a,\boldsymbol\epsilon_1
 + \boldsymbol\epsilon_4\rangle
  \neq 0  \right\}.
  \]
These two subsets are orbits for the action of the Euclidean group
$\SO(3) \rtimes \R^3$, but are not preserved by the action of the
full Laguerre group.
A Euclidean motion $(R;v)\in \SO(3) \rtimes \R^3$ maps the sphere $S_r(p)$ into the sphere $S_r(p')$ with the same radius
$r$ and center $p' = Rp+v$. A time translation $\tau_s$ maps the sphere $S_r(p)$ into the sphere $S_{r'}(p)$ with the
same center $p$ and signed radius $r' = r +s$. Finally, a pure Lorentz transformation $B_b$, $ b\in \R^3$,
${}^t\!b b<1$, takes the sphere $S_r(p)$ to the sphere $S_{r'}(p')$, where
\[
 r' = \beta(r-{}^t\!b p), \quad p' = p +
  \left((\beta- 1)\frac{{}^t\!p b}{{}^t\!  b  b}  - \beta r\right)  b.
\]

\subsubsection{The space of oriented planes}
Next, we examine the 3-dimensional orbit $\mathcal Q_\Pi$. The orientation of a
2-plane $\pi \subset \R^3$ is given by fixing a unit normal vector. For $h\in \R$, $p\in \R^3$, and $n={}^t\!(n^1,n^2,n^3)\in S^2$,
let $\Pi_h(n)$ and $\Pi_p(n)$
denote, respectively, the oriented plane with unit normal $n$ and height $h$ given by
\[
  \Pi_h(n) = \left\{x\in \R^3 \, : \, {}^t\! n x = h\right\},
  \]
and the oriented plane through the point $p$ normal to $n$ given by
\[
 \Pi_p(n) = \left\{x\in \R^3 \, : \, {}^t\! n(x-p) = 0\right\}.
  \]

The oriented plane $\Pi_p(n)$ can be uniquely represented by the null line $[\mathfrak b (n,p)]$ spanned by the
null vector
\begin{equation}\label{b-n-p}
 \mathfrak b(n,p) = {}^t\!\Big(0, \frac{1 +n^1}{2},\frac{n^2}{\sqrt2},\frac{n^3}{\sqrt2}, \frac{1 -n^1}{2}, \frac{{}^t\!np}{\sqrt2}\Big).
 \end{equation}
 Thus the points of $\mathcal Q_\Pi$ represent the oriented planes of Euclidean space.

 \subsubsection{Oriented contact}
Two oriented planes $\Pi_h(n)$ and $\Pi_{h'}(n')$ are in {\it oriented contact} if they have the same
unit normals, i.e., $n = n'$. Two oriented spheres $S_r(p)$ and $S_{r'}(p')$ are in {\it oriented contact}
if the Euclidean distance $d(p, p')$ between $p$ and $p'$ coincides with $|r-r'|$, that is, $d(p, p') = |r-r'|$.
This amounts to saying that
the Lorentz scalar product $\left(V(r,p)- V(r',p'), V(r,p)- V(r',p')\right) = 0$. The sphere $S_r(p)$ and the plane
$\Pi_h(n)$ are in oriented contact if $n^1p^1 +n^2p^2+ n^3p^3= h +r$. In particular, a point sphere
$S_0(p)$ and an oriented plane $\Pi_h(n)$ are in oriented contact if $p \in \Pi_h(n)$.

From the above identifications, it follows that two null lines $[\mathfrak a]$ and $[\mathfrak b]$ represent
oriented spheres or oriented planes in oriented contact if and only if $ \langle\mathfrak a, \mathfrak b  \rangle =0$.

\subsection{The Laguerre space}

We let $\boldsymbol\Lambda$ denote the set of all 2-dimensional subspaces
$\lambda(\mathfrak a, \mathfrak b)$ of $\R^6$ spanned by vectors $\mathfrak a$ and $\mathfrak b$ satisfying
\[
  \langle\mathfrak a, \mathfrak a  \rangle = \langle\mathfrak a, \mathfrak b  \rangle
  = \langle\mathfrak b, \mathfrak b  \rangle = 0, \quad \boldsymbol\epsilon_5 \notin \lambda(\mathfrak a, \mathfrak b).
   \]
We call $\boldsymbol\Lambda$ the {\it Laguerre space}. This is a 5-dimensional orbit
of the action of
the Laguerre group $L$ on the Grassmannian $G_2(6)$ of 2-planes in $\R^6$, where the action is given
by
\[
A \lambda(\mathfrak a, \mathfrak b) := \lambda(A \mathfrak a, A\mathfrak b), \quad \forall\, A \in L,\, \forall\,
\mathfrak a, \mathfrak b \in \R^6.
\]

\begin{remark}
The Laguerre space $\La$ can be seen as the space of projective lines in the Lie quadric
$\mathcal Q = \left\{[\mathfrak a] \, : \, \langle \mathfrak a, \mathfrak a \rangle =0\right\}$
which do not meet the the ``point at infinity'' $\{[\boldsymbol\epsilon_5]\}$.
The Laguerre space $\La$ is a dense open set of the set of all lines in $\mathcal Q$.
It is the complement of the set of lines through $[\boldsymbol\epsilon_5]$ in $\mathcal Q$ (cf. \cite{J}).\footnote{Notice
that in the lecture notes of G. R. Jensen \cite{J},
the symbol $\La$ is used to denote the set of all lines in $\mathcal Q$, while the set of lines in $\mathcal Q$
which do not meet the point at infinity is denoted by $\La'$.}
\end{remark}

We will provide two geometrical realizations of the Laguerre space $\boldsymbol\Lambda$.

\subsubsection{The Laguerre space as the unit tangent bundle of $\R^3$}
The Laguerre space $\boldsymbol\Lambda$ is identified with the bundle $T_1(\R^3) \cong \R^3 \times S^2$
of unit tangent vectors of $\R^3$ by means of the map
\begin{equation}\label{lag-space-utg}
  \R^3 \times S^2 \ni (p,n) \longmapsto \lambda(\mathfrak a(p), \mathfrak b(n,p)) \in \boldsymbol\Lambda,
  \end{equation}
where $\mathfrak a(p) = \mathfrak a(0,p)$ and $\mathfrak b(n,p)$ are given, respectively, as in \eqref{a-r-p} and \eqref{b-n-p}.

\subsubsection{The Laguerre space as the set of parabolic pencils of $\R^3$}

Let $\Pi_h(n)$ be an oriented plane and let $p$ be a point on this plane. The set $\widehat \lambda(p,n)$ of all oriented spheres
in oriented contact with $\Pi_h(n)$ at $p$ is called the {\it parabolic pencil} of spheres with base locus $p$ and radical
plane $\Pi_h(n)$.

  If $\widehat\lambda(p,n)$ is a parabolic pencil with base locus $p$ and radical
plane $\Pi_h(n)$, the null vectors $\mathfrak a(0,p)$ and $\mathfrak b(n,p)$ (cf. \eqref{a-r-p} and \eqref{b-n-p})
span a null plane and the oriented spheres of the pencil are represented by the null lines lying on the null
plane spanned by $\mathfrak a$ and $\mathfrak b$. {\it The Laguerre space can thus be interpreted as the set of
all parabolic pencils of oriented spheres in $\R^3$}.

Given a null plane $\lambda(\mathfrak a, \mathfrak b)\in \boldsymbol\Lambda$, there exists a
unique $[\hat{\mathfrak a}] \subset \lambda(\mathfrak a, \mathfrak b)$ so that
$\langle\hat{\mathfrak a}, \boldsymbol\epsilon_1 + \boldsymbol\epsilon_4  \rangle = 0$. Thus
$[\hat{\mathfrak a}]$ represents a point $p(\mathfrak a, \mathfrak b) \in \R^3$, which we call
the {\it Euclidean projection} of $\lambda(\mathfrak a, \mathfrak b)$. Using \eqref{lag-space-utg}, it follows that
the Euclidean projection $p(\mathfrak a, \mathfrak b)$ coincides with the
first component in $(p,n)$, where $(p,n)$
is the contact element which represents the null plane $\lambda(\mathfrak a, \mathfrak b)$.

In other words, {\it the Euclidean projection
$\boldsymbol\Lambda \to \R^3$, $\lambda \mapsto p(\lambda)$ agrees with the bundle projection map
$\R^3 \times S^2 \to \R^3$}. Observe, however, that the Laguerre group $L$ does not preserve the Euclidean projection.

\subsection{The structure equations of the Laguerre space}
The Laguerre group $L$ acts transitively on $\boldsymbol\Lambda$ and the map
\begin{equation}
 \pi_L : L \to \boldsymbol\Lambda, \quad A \longmapsto \lambda(A_0, A_1) = A\lambda(\boldsymbol\epsilon_0,
 \boldsymbol\epsilon_1)
  \end{equation}
makes $L$ into a principal bundle over $\boldsymbol\Lambda$ with structure
group
\[
 L_{0} = \{A = (A^i_j) \in L \, : \, A^j_0 = A^j_1 = 0, \, j = 2,3,4,5 \}.
  \]
The elements of $L_0$ are matrices of the form
\begin{equation}\label{elements-L}
X(d;b;x)=
 \begin{bmatrix}
 1   & 0  & 0         & 0         & 0                     & 0 \\
 d_1 &d_2 &\tilde x^1 &\tilde x^2 &{\frac{d_2}{2}{}^t\!xx} & 0\\
 0   & 0  &b^1_1      &b^1_2      &x^1                    & 0\\
 0   & 0  &b^2_1      &b^2_2      &x^2                    & 0\\
 0   & 0  & 0         &  0        &1/d_2                  &0 \\
 0   & 0  & 0         &  0        &-d_1/d_2               &1 \\
\end{bmatrix},
\end{equation}
where $b=(b_j^i)\in \SO(2)$, $d=(d_1,d_2)$, $d_2 \neq 0$,
$x= {}^t\!(x^1,x^2)\in\R^2$, $(\tilde x^1,\tilde x^2)=d_2{}^t\!xb$,
and $\frac{1}{d_2}(1 + d_2^2) + \frac12d_2 {}^t\!x x \geq 2$.

A local {\it Laguerre frame field} is
a local cross section $A : U \subset \boldsymbol\Lambda  \to L$ of
the {\it Laguerre fibration} $\pi_L$, where $U$ is an open subset
$U$ of $\boldsymbol\Lambda$. If $A : U \to L$ is a Laguerre frame field, any other Laguerre
frame field $\hat A$ on $U$ is given by
\begin{equation}\label{frame-gauge}
 \hat A = A X(d;b;x),
   \end{equation}
 where $X(d;b;x) : U \to L_0$
is a smooth map. If we set
\[
 \theta = (\theta^i_j), \quad \theta^i_j = A^\ast(\omega^i_j),
  \]
then the 1-forms
\begin{equation}\label{coframe}
   \theta^2_0,\,\theta^3_0,\, \theta^4_0,\, \theta^2_1,\, \theta^3_1
   \end{equation}
define a local coframe on the open subset $U\subset \boldsymbol\Lambda$.
If $A$ and $\hat A$ are related by \eqref{frame-gauge},
the corresponding $\mathfrak l$-valued 1-forms $\theta$ and $\hat \theta$ are related by
\begin{equation}\label{transf-rule}
 \hat \theta =  X^{-1}\theta X + X^{-1} dX.
   \end{equation}
In particular, the coframe \eqref{coframe} is subject to the following transformation rules
\begin{equation}\label{transf-rule-coframe}
\left[\begin{array}{c} \hat \theta^2_0\\ \hat \theta^3_0\\ \hat \theta^4_0\\ \hat \theta^2_1\\ \hat \theta^3_1
\end{array}\right] =
\left[\begin{array}{c|c|c} {}^t\!b & -d_2{}^t\!b x & d_1{}^t\!b\\
\hline
0 & d_2 & 0 \\
\hline
 0 & 0 & d_2 {}^t\!b
\end{array}\right]
\left[\begin{array}{c}  \theta^2_0\\  \theta^3_0\\  \theta^4_0\\  \theta^2_1\\  \theta^3_1
\end{array}\right].
\end{equation}

According to the transformation rules \eqref{transf-rule-coframe},
there are three naturally defined subbundles of the tangent bundle $T(\boldsymbol\Lambda)$ of the Laguerre
space:

\begin{enumerate}
\item the subbundle $\mathcal Z \subset T(\boldsymbol\Lambda)$, defined by requiring that $\theta^4_0 =0$;
\item the subbundle $\mathcal H \subset T(\boldsymbol\Lambda)$, defined by requiring that $\theta^2_1 =\theta^3_1=0$;
\item the subbundle $\mathcal K \subset T(\boldsymbol\Lambda)$, defined by requiring that $\theta^4_0 =\theta^2_1 =\theta^3_1=0$.
\end{enumerate}
Notice that $\text{rank}(\mathcal Z) = 4$, $\text{rank}(\mathcal H) = 3$, $\text{rank}(\mathcal K) = 2$, and that
\[
 \mathcal K = \mathcal Z \cap \mathcal H.
  \]
  From the structure equations \eqref{str-eq2}, it follows that
  \[
   d\theta^4_0 \equiv \theta^2_0 \wedge \theta^2_1 + \theta^3_0 \wedge \theta^3_1 \mod \{\theta^4_0\}.
  \]
  This implies that $d\theta^4_0 \wedge d\theta^4_0 \wedge \theta^4_0$ is never zero on $\boldsymbol\Lambda$,
  and hence $\theta^4_0$ defines a {\it contact structure} on the Laguerre space $\boldsymbol\Lambda$.
The subbundle $\mathcal Z$ is a {\it contact distribution} on
$\boldsymbol\Lambda$.

\begin{remark}
Observe that the 1-form $\theta^4_0=-\langle dA_0,A_1 \rangle$. Therefore,
the 1-form $\theta^4_0$ corresponds via
the map \eqref{lag-space-utg} to the 1-form $\frac{1}{\sqrt2} \sum_{i=1}^3 n^i dp^i = \frac{1}{\sqrt2} dp\cdot n$.
Thus, the contact structure defined by $\theta^4_0$
coincides with
the natural contact structure on $\R^3 \times S^2$ given by the 1-form
$\alpha_{(p,n)} = \sum_{i=1}^3  n^idp^i = dp\cdot n$.
\end{remark}

\section{Moving frames for Legendre immersions}\label{s:nondeg-Leg-imm}

\begin{defn}\label{legendre-surf}
Let $S$ be an oriented, connected, 2-dimensional manifold. A smooth immersion $f : S \to \boldsymbol\Lambda$ is
called {\it Legendre} if
\begin{equation}\label{legendre-cond}
 df_{|s}(T_s S)
 \subset \mathcal Z_{f(s)}, \quad \forall\, s \in S.
  \end{equation}
  This amounts to saying that $f^\ast (\theta^4_0) = 0$ on $S$.
  \end{defn}

Two Legendre immersions $f,f' : S \to \La$ are said to be $L$-{\it equivalent}
if there exists $A\in L$, such that $Af'(S) = f(S)$. In this case, the two immersions $f$ and $f'$
are considered to be the same geometric object.

\begin{remark}
An immersed surface $x : S \to \R^3$, oriented by a unit normal
field $n : S \to S^2$, induces a lift $f= (x,n): S \to \La$ of $x$ to
$\boldsymbol\Lambda\cong \R^3 \times S^2$ which
is a Legendre immersion, since $dx \cdot n=0$ . We call $f$ the {\it Legendre lift} of
$x$. Locally and up to $L$-equivalence, any Legendre immersion arises in this way.
However, observe that in general, if
$f = (x,n) : S \to \boldsymbol\Lambda\cong \R^3 \times S^2$ is a Legendre immersion, $x$ need not be an immersion into
$\R^3$.
In particular, two immersions $x,x' : S\to \R^3$ are $L$-equivalent if their
Legendre lifts are $L$-equivalent.

For a generic Legendre immersion $f = (x,n)$, the quadratic form $dn\cdot dn$ is positive semidefinite,
that is $dn\cdot dn\geq 0$.
\end{remark}

\begin{defn}\label{nondegenerate}
A Legendre immersion $f = (x,n) : S \to \La$ is said to be {\it nondegenerate} if (1)
the quadratic form $dn\cdot dn$ is positive definite and (2) the quadratic forms $dx\cdot dn$ and
$dn\cdot dn$ are everywhere linearly independent on $S$.
\end{defn}

If $f =(x,n)$ is the Legendre lift  of an immersion $x : S\to \R^3$, the condition that $f$ is
nondegenerate
means that $x$ is umbilic free and its Gauss curvature is everywhere different from zero.

The condition that $dn\cdot dn$ is positive definite will be assumed throughout.

\subsection{Construction of the canonical frame}\label{ss:construction}

Let $f : S \to \La$ be a Legendre immersion. A local {\it Laguerre frame field} along $f$
is a smooth map $A : U \subset S \to L$ defined on an open subset $U$ of $S$ such that
$\pi_L ([A(s)]) = f(s)$, for every $s\in U$.
If $f : S \to \La$ is a Legendre immersion and
$A : U \to L$ is a Laguerre frame field on $U$, then
\[
  \alpha^4_0 = 0 , \quad \alpha^2_1 \wedge \alpha^3_1 \neq 0,
  \]
where $\alpha^i_j := A^\ast(\omega^i_j)$. Differentiating $\alpha^4_0 = 0$ and using
\eqref{str-eq2} yields
\[
 \alpha^2_0 \wedge \alpha^2_1 + \alpha^3_0 \wedge \alpha^3_1 = 0.
  \]
By Cartan's Lemma, there are smooth functions $h_{ij} =h_{ji} : U \to \R$, $i,j = 1,2$, which depend on $A$,
such that
\begin{equation}\label{alfa2030-2131}
 \begin{array}{c}
  \alpha^2_0 = h_{11}\alpha^2_1 + h_{12}\alpha^3_1 \\
   \alpha^3_0 = h_{12}\alpha^2_1 + h_{22}\alpha^3_1.
     \end{array}
      \end{equation}
Any other frame field $\hat A$ on $U$ is given by $$\hat A = A X(d;b;x),$$ where $X =X(d;b;x) : U \subset S \to L_0$
is a smooth map. By \eqref{transf-rule}, $\hat \alpha$ and $\alpha$ are related by
$$
  \hat \alpha = X^{-1}\alpha X + X^{-1}dX.
   $$
This implies
\begin{equation}\label{transf-rule-alfa2131}
\left[\begin{array}{c}  \hat \alpha^2_1\\ \hat \alpha^3_1 \end{array}\right]
= d_2 {}^t\!b
\left[ \begin{array}{c}   \alpha^2_1\\  \alpha^3_1 \end{array}\right],
\end{equation}
and
\begin{equation}\label{transf-rule-alfa2030}
\left[\begin{array}{c}
\hat \alpha^2_0\\ \hat \alpha^3_0
\end{array}\right]
= {}^t\!b\left(\left[\begin{array}{c}   \alpha^2_0\\  \alpha^3_0 \end{array}\right]
+ d_1  \left[ \begin{array}{c}  \alpha^2_1\\  \alpha^3_1 \end{array}\right]
\right).
\end{equation}
From \eqref{alfa2030-2131}, \eqref{transf-rule-alfa2131}, and \eqref{transf-rule-alfa2030}, it follows
that
\begin{equation}\label{transf-rule-hij}
\left[\begin{array}{cc}  \hat h_{11} & \hat h_{12}\\ \hat h_{12} & \hat h_{22} \end{array}\right]
= \frac{1}{d_2} {}^t\!b \left[ \begin{array}{cc} h_{11} +d_1&  h_{12}\\ h_{12} &  h_{22} +d_1 \end{array}\right]b.
\end{equation}
Therefore,
\begin{equation}\label{h11+h22}
 \left(\hat h_{11} + \hat h_{22}\right) =
  \frac{1}{d_2}\left[  (h_{11} +   h_{22}) + 2d_1\right].
   \end{equation}

This shows that locally there exist Laguerre frame fields such that
\begin{equation}\label{2-order-cond}
  h_{11} +   h_{22}=0.
   \end{equation}

\begin{defn}
A Laguerre frame field $A : U \to L$ along $f$ is said to be of {\it first order} if $h_{11} +  h_{22}=0$ on $U$.
\end{defn}

It follows from \eqref{h11+h22} that if $\hat A$ and $A$ are first order frame fields on $U$, then
$\hat A = A X(d;b;x)$ and $d_1 = 0$. This means that the totality of first order frame fields
defines an $L_1$-fibre bundle
on $S$ whose structure group is
\[
  L_1 = \{X(d;b;x) \in L_0 \, : \; d_1 =0\}.
   \]

\begin{remark}
A {\it sphere congruence} is a 2-parameter family of oriented spheres, i.e., a
smooth map $\sigma : S \to \mathcal Q_\Sigma$ of a connected surface into the space of oriented spheres.
By an envelope of $\sigma : S \to \mathcal Q_\Sigma$ we mean a Legendre map $f = (x,n) : S \to \La$
such that $\sigma(s)$ and the plane $\Pi_{x(s)}(n(s))$ are in oriented contact at $x(s)$, for every
$s \in S$.
\end{remark}

\begin{defn}
Let $f : S \to \La$ be a Legendre immersion and let $A : U \subset S \to L$ be a first order frame field
along $f$.
The smooth map $U\ni s \mapsto [A_0(s)] \in \mathcal Q_\Sigma$
is independent of the choice of first order frame field $A$. This allows the definition
of a smooth map on the whole $S$, $$\sigma_f : S \to \mathcal Q_\Sigma.$$
The map $\sigma_f$ determines a sphere congruence
which is known as the {\it middle congruence} (cf. \cite{Blaschke}) or {\it Laguerre Gauss map} of $f$. By construction,
the Legendre immersion $f: S \to \La$ is an envelope of the middle congruence.
\end{defn}

According to \eqref{transf-rule-alfa2030}, for a change of first order frames $\hat A = A X$, where $X : U \to L_1$,
we have
\begin{equation}\label{alfa2030-2131-bis}
\left[\begin{array}{c}  \hat \alpha^2_0\\ \hat \alpha^3_0 \end{array}\right]
 = {}^t\!b \left[ \begin{array}{c}   \alpha^2_0\\  \alpha^3_0 \end{array}\right].
   \end{equation}
This implies that the Legendre immersion $f$ induces on $S$ a {\it quadratic form}
\[
  \Phi_f = (\alpha^2_0)^2 + (\alpha^3_0)^2,
   \]
the {\it Laguerre metric}
of the immersion, and an exterior differential 2-form
\[
  \Omega_f = \alpha^2_0\wedge \alpha^3_0,
   \]
the {\it Laguerre area element} of the immersion.

\begin{remark}\label{rk:sigma-f-spacelike}
If the Legendre immersion  $f : S  \to \La$ is {nondegenerate}, the form $\Phi_f$ is positive
definite and ${\Omega_f}\vert_s \neq 0$, for every $s\in S$.
In particular, note that $\Phi_f$ is the metric induced on $S$ by the Laguerre Gauss
map $\sigma_f$ from the Lorentz product $\langle\,,\rangle$ on $\mathcal Q_\Sigma \cong \R^{3,1}$, i.e.,
$\Phi_f = \langle d\sigma_f, d\sigma_f\rangle$. Thus, if $f$ is nondegenerate, its Laguerre Gauss map
$\sigma_f$ defines a spacelike immersion of $S$ into Minkowski 4-space $\R^{3,1}$
and ${\Omega_f}$ is the induced area element of $\sigma_f :S\to \mathcal Q_\Sigma$.

\end{remark}


Let $f : S \to \La$ be a nondegenerate Legendre immersion and let $A : U \subset S \to L$ be a first
order frame field along $f$. Then there exist smooth functions $a_1, a_2 : U \to \R$ depending on $A$ such that
\begin{equation}\label{alfa10}
 \alpha^1_0 = a_1 \alpha^2_0 + a_2 \alpha^3_0.
   \end{equation}

   If $\hat A = A X(d;b;x)$ is any first order frame field on $U$, where $X : U \to L_1$, it follows from
   \eqref{transf-rule} that
\begin{equation}\label{alfa10-bis}
 \hat \alpha ^1_0 = \frac{1}{d_2}\left(\alpha^1_0 -d_2 x^1 \alpha^2_0 -d_2 x^2 \alpha^3_0  \right).
  \end{equation}
If we write  $\hat \alpha^1_0 = \hat a_1 \hat \alpha^2_0 + \hat a_2 \hat \alpha^3_0$, \eqref{alfa2030-2131-bis}
and \eqref{alfa10-bis} yield
\begin{equation}\label{a1a2-trans}
\left[\begin{array}{c}  \hat a_1\\ \hat a_2 \end{array}\right]
 = \frac{1}{d_2}{}^t\!b \left[ \begin{array}{c}   a_1 - d_2x^1\\  a_2 -d_2x^2 \end{array}\right].
   \end{equation}

From the above transformation formula, it follows that, about any point $s\in S$, there exists
an open neighborhood $U$ and a Laguerre frame field of first order defined on $U$
 for which $a_1 = a_2 = 0$, i.e., $\alpha^1_0 =0$.

\begin{defn}
A first order Laguerre frame field $A : U \to L$ along $f$ is said to be of {\it second order} if $\alpha^1_0 = 0$ on $U$.
\end{defn}

It follows from \eqref{a1a2-trans} that if $\hat A$ and $A$ are second order frame fields on $U$, then
$\hat A = A X(d;b;0)$ and $d_1 = 0$. This means that the totality of second order frame fields generates
an $L_2$-fibre bundle
on $S$ whose structure group is
\[
  L_2 = \{X(d;b;x) \in L_1 \, : \; x ={}^t\!(0,0)\}.
   \]

\begin{remark}[The Laguerre transform]\label{rk:lag-transf}
For a change of second order frames $\hat A = A X$, where $X : U \subset S \to L_2$, it is easily seen
that the projective lines $\lambda(A_0,A_4)$ and $\lambda(\hat A_0,\hat A_4)$ do coincide.
Accordingly, we can define a new Legendre immersion (possibly degenerate)
\[
 \check f : S \to \La, \quad \check f(s) := \lambda(A_0(s),A_4(s)), \quad \forall\, s \in S.
  \]
  The Legendre immersion $\check f$ is called the {\it Laguerre transform} of $f$. The Laguerre
  transform $\check f$ is the second envelope of the middle sphere congruence $\sigma_f$.
\end{remark}

From \eqref{transf-rule-hij} it follows that about any point $s\in S$ there exists
an open neighborhood $U$ and a Laguerre frame field of second order defined on $U$
 for which $h_{11} = - h_{22} = 1$ and $h_{12} = 0$.

\begin{defn}
A second order Laguerre frame field  $A : U \to L$ along $f$
for which $h_{11} = - h_{22} = 1$ and $h_{12} = 0$ on $U$ is called a {\it canonical} Laguerre frame field.
\end{defn}

According to \eqref{transf-rule-hij}, two canonical frames $\hat A, A : U \to L$  are related
by $\hat A = A X(d;b;0)$, where $X$ is a smooth map taking values in the group
\[
  \mathbb Z_2 = \left\{X(d;b;x) \in L_2 \, : \; d_2 =1, \, b = \pm  \begin{bmatrix} 1 & 0\\0&1\end{bmatrix}\right\}.
   \]
If $A = (A_0,A_1,A_2,A_3,A_4,A_5) : U\subset S \to L$ is a canonical frame along $f : S \to \La$ and
$U$ is connected, the only other canonical frame on $U$ is given by
\begin{equation}\label{other-canonical}
 \hat A = \left(A_0,A_1, -A_2, -A_3, A_4, A_5\right).
   \end{equation}
A canonical frame $A = (A_0,A_1,A_2,A_3,A_4,A_5)$ is characterized by the following equations
\begin{equation}\label{dA=Aalpha}
\begin{cases}
dA_0 = \alpha^2_0 A_2 +\alpha^3_0 A_3,\\
dA_1 = \alpha^1_1 A_1 +\alpha^2_0 A_2 - \alpha^3_0 A_3,\\
dA_2 = \alpha^1_2 A_1 +\alpha^3_2 A_3 + \alpha^2_0 (A_4 + A_5), \\
dA_3 = \alpha^1_3 A_1 -\alpha^3_2 A_2 + \alpha^3_0 (-A_4 + A_5),\\
dA_4 = \alpha^1_2 A_2 +\alpha^1_3 A_3 - \alpha^1_1 A_4, \\
dA_5 =0.
\end{cases}
\end{equation}

The totality of canonical frame fields on $S$ defines
a $\mathbb Z_2$-fibre bundle $\pi_f : \mathcal F(f) \to S$,
where
\[
  \mathcal F(f) = \{(s,A(s)) \in S \times L\}
   \]
and $A$ is a canonical frame field along $f$ defined on a neighborhood of $s$,
and $\pi_f : \mathcal F(f) \ni (s, A(s))  \mapsto  s\in S$.

\begin{remark}\label{r:global-can-fr}
Up to $L$-equivalence, and a 2:1 covering, any nondegenerate Legendre
immersion
of an oriented surface $S$
admits a globally defined canonical frame.
In fact, consider a connected component $\tilde S$ of $\mathcal F(f)$ and let $\tilde f : \tilde S \to \La$
be given by $\tilde f = f\circ \pi_f$. By construction, $\tilde f$ is again a nondegenerate Legendre immersion
and $f(S) = \tilde f(\tilde S)$. Moreover, $ \tilde A : \tilde S \to L$, $(s,A(s))  \mapsto A(s)$, is a global
canonical frame along $\tilde f : \tilde S \to \La$.
\end{remark}

\subsection{The structure equations of the canonical frame}

Let $f : S \to \La$ be a nondegenerate Legendre immersion and let $A : S \to L$ be a
canonical frame field along $f$. According to Remark \ref{r:global-can-fr}, it is not restrictive
to assume that $A$ globally defined on $S$.
From the above discussion, $(\alpha^2_0,\alpha^3_0)$
defines a coframe on $S$, referred to as the {\it canonical coframe}, and
\begin{equation}\label{alfa21-31}
 \alpha^2_1 = \alpha ^2_0, \quad \alpha^3_1 = -\alpha ^3_0, \quad \alpha^1_0 = 0, \quad \alpha^4_0 =0.
  \end{equation}
  Differentiating $\alpha^1_0 = 0$ yields
\[
  \alpha^1_2 \wedge \alpha ^2_0 + \alpha^1_3 \wedge \alpha ^3_0 = 0.
 \]
Thus, by Cartan's Lemma there
exist smooth functions $p_1$, $p_2$, $p_3 : S \to \R$
such that
\begin{equation}\label{alfa12-13}
  \alpha^1_2 =  p_1\alpha^2_0 + p_2\alpha^3_0, \quad \alpha^1_3 = p_2\alpha^2_0+p_3\alpha^3_0.
   \end{equation}
Next, write
\begin{equation}\label{alfa32}
  \alpha^3_2 =  q_1\alpha^2_0 + q_2\alpha^3_0,
   \end{equation}
where $q_1$, $q_2 : S \to \R$ are smooth functions. We call the functions  $p_1$, $p_2$, $p_3$, $q_1$, $q_2$
the {\it invariant functions} of the Legendre immersion $f$ (with respect to the canonical frame field $A$).
Exterior differentiation of $\alpha^2_0$ and $\alpha^3_0$, using the structure equations \eqref{str-eq2} and
equation \eqref{alfa32}, yields
\begin{equation}\label{se0}
 d\alpha^2_0 = q_1\alpha^2_0\wedge\alpha^3_0, \quad  d\alpha^3_0=q_2\alpha^2_0\wedge\alpha^3_0.
  \end{equation}
Differentiating $\alpha^2_1$ and $\alpha^3_1$, using the structure equations \eqref{str-eq2} and
the equations \eqref{se0}, yields
\begin{equation}\label{alfa11}
  \alpha^1_1 =  2q_2\alpha^2_0 -2 q_1\alpha^3_0.
  \end{equation}
From the structure equations \eqref{str-eq2}, we also obtain the following equations:
\begin{eqnarray}
dq_1\wedge\alpha^2_0+dq_2\wedge\alpha^3_0&=&
(p_3-p_1-{q_1}^2-{q_2}^2)\alpha^2_0\wedge\alpha^3_0,\label{se1}\\
dq_1\wedge\alpha^3_0-dq_2\wedge\alpha^2_0&=& - p_2\alpha^2_0\wedge\alpha^3_0,\label{se2}\\
dp_1\wedge\alpha^2_0+dp_2\wedge\alpha^3_0&
= &(-3q_1p_1-4q_2p_2+q_1p_3)\alpha^2_0\wedge\alpha^3_0,\label{se3}\\
dp_2\wedge\alpha^2_0+dp_3\wedge\alpha^3_0&=&
(-3q_2p_3-4q_1p_2+q_2p_1)\alpha^2_0\wedge\alpha^3_0.\label{se4}
\end{eqnarray}
The equations \eqref{se0}, \eqref{se1}, \eqref{se2}, \eqref{se3} and \eqref{se4} will be referred
to as the {\it structure equations of the Legendre immersion}
$f$.
The existence of a canonical frame
field along $f$ under the nondegeneracy assumption was first proved in \cite{MN-TAMS}.
The smooth functions $q_1$,
$q_2$, $p_1$, $p_2$, $p_3$ form a complete system of Laguerre
invariants for $f$.


\begin{remark}\label{r:p-q-change}
If ${A} : S \to L$ is a canonical frame field along $f$,
the only other canonical frame field is given by $\hat A$ as in \eqref{other-canonical}.
Under this frame change, the invariants $p_1, p_2, p_3, q_1, q_2$ transform by
$$
\hat q_1= -q_1,\quad   \hat q_2= -q_2,\quad
\hat p_1= p_1, \quad \hat p_2= p_2,  \quad \hat p_3= p_3.
$$
Thus, there are well defined global functions $\text{\sc j}$,
$\text{\sc w}: S \to \R$ such that
locally
\begin{equation}\label{mixed-invariants}
 \text{\sc j} =\frac{1}{2}(p_1-p_3), \quad
  \text{\sc w}= \frac{1}{2}(p_1+p_3).
   \end{equation}
\end{remark}


\subsection{Framed Legendre immersions and the Laguerre Pfaffian system}\label{ss:LPS}

In this section, nondegenerate Legendre immersions will be interpreted as integral manifolds of a quasi-linear
Pfaffian differential system.

By a {\it framed Legendre immersion} $(S,f,A)$ is meant a nondegenerate Legendre immersion $f: S \to \La$ endowed with
a canonical frame $A : S \to L$. The immersion $f$ is then given by $\pi_L \circ A$ and the normal frame field $A$ satisfies
the conditions
\begin{equation}\label{framed-imm-conds}
 \alpha^2_0\wedge \alpha^3_0 \neq 0, \quad \alpha^2_1 -\alpha^2_0 = \alpha^3_1 +\alpha^3_0 =\alpha^1_0 =
 \alpha^4_0 = 0.
\end{equation}
Conversely, if $A : S \to L$ is a smooth map satisfying \eqref{framed-imm-conds}, then $f = \pi_L \circ A : S \to\La$
is a nondegenerate Legende immersion and $A$ is a canonical frame field along $f$. Thus framed Legendre immersions can be interpreted as the integral manifolds of the Pfaffian differential system on $L$ defined by the equations
\begin{equation}\label{PS-framed-imm}
  \omega^2_1 -\omega^2_0, \quad \omega^3_1 +\omega^3_0 =0, \quad  \omega^1_0 =0, \quad
 \omega^4_0 = 0,
\end{equation}
with independence condition $\omega^2_0\wedge \omega^3_0 \neq 0$. This Pfaffian differential system
is not in involution. Its involutive prolongation is given by the following system.

Let $M := L \times \mathbb R^5$ and denote by $(q_1,q_2, p_1, p_2, p_3)$ the coordinates in $\mathbb R^5$.
Let $\omega^1 = \omega^2_0$, $\omega^2 = \omega^3_0$, and define on $M$ the absolute parallelism given by the coframe
field $(\omega^i, \eta^a, \pi^i, \zeta^s)$, $a = 1,\dots,8$; $i =1,2$; $s=1,2,3$, where
\[
\begin{aligned}
&\eta^1 = \omega^4_0, \quad \eta^2 = \omega^1_0, \quad \eta^3 = \omega^2_1-\omega^1, \quad \eta^4 = \omega^3_1+\omega^1,\\
&\eta^5 = \omega^3_2 -q_1\omega^1 -q_2\omega^2, \quad \eta^6 = \omega^1_1 -2q_2\omega^1 +2q_1\omega^2, \\
&\eta^7 = \omega^1_2 -p_1\omega^1 -p_2\omega^2, \quad \eta^8 = \omega^1_3 -p_2\omega^1 -p_3\omega^2, \\
&\pi^i = dq_i, \quad \zeta^s = dp_s.
\end{aligned}
\]

Let $\mathcal I$ be the ideal of the algebra of exterior differential forms on $M$ generated by
\[
 \eta^1, \dots, \eta^8, d\eta^1, \dots, d\eta^8.
   \]
%
Let $(\mathcal I, \Omega)$ be the Pfaffian differential system on $M$ defined
by the differential ideal $\mathcal I$ with the independence condition
\[
 \Omega = \omega^1 \wedge \omega^2 \neq 0.
  \]
%
%
If $(A,q_1,q_2, p_1, p_2, p_3): S \to M$ is an integral manifold of $(\mathcal I,\Omega)$, then $f : S \to \La$,
$s\mapsto f(s) = \pi_L(A(s))$, defines a Legendre immersion with canonical frame field $A$
and corresponding invariant functions $q_1,q_2, p_1, p_2, p_3$. Conversely, any framed Legendre
immersion $(S,f,A)$ defines an integral manifold of the Pfaffian system $(\mathcal I,\Omega)$ by
\[
  S\ni s \longmapsto (A(s),q_1,q_2, p_1, p_2, p_3) \in M,
   \]
   where $(q_1,q_2, p_1, p_2, p_3)$ are the invariant functions of $A$.
Summarizing, framed Legendre immersions, together with the associate invariant functions, may be
regarded as being the integral manifolds of the Pfaffian
system $(\mathcal I,\Omega)$ on $M$.

\begin{defn}
The Pfaffian system $(\mathcal I,\Omega)$
on $M$ will
be referred to as the {\it Laguerre differential system} and will be denoted by $(M, \mathcal I,\Omega)$.
\end{defn}

\subsection{Relations with Euclidean geometry}\label{ss:euclidean-inv}

Let $f = (x,n) : S \to \La$ be the
Legendre lift of an oriented immersion $x : S \to \mathbb R^3$ with unit normal field $n$
and suppose that $x$
is umbilic free and that its Gauss curvature is everywhere different from zero.
%
We may assume the existence of a global principal frame field $e = (n,e_2,e_3 ;x) : S \to \mathbb E(3)$
along $x$, such that, at every point $s\in S$,
$(e_2 ,e_3)_{|s}$ is a positive basis of $dx_{|s} (T_s S)$ and $e_2 ,e_3$
are along the principal directions.
The frame field $e$ satisfies the equations
\[
 dx = \varphi^2 e_2 + \varphi^3 e_3, \quad de_2 = \varphi^3_2e_3 -\varphi^2_1n,
 \quad de_3 = -\varphi^3_2 e_2 - \varphi^3_1 n,
\]
and
\[
  dn = \varphi^2_1 e_2 +\varphi^3_1 e_3,
\]
where $(\varphi^2, \varphi^3)$ is the dual coframe of $(e_2,e_3)$.
The condition that $e_2,e_3$ are along principal directions is expressed by
\[
 \varphi^2_1 = -a \varphi^2, \quad \varphi^2_1 = -c \varphi^3,
\]
where $a>c$ are the principal curvatures.

By \eqref{euc-group-incl} and \eqref{exp-euc-group-incl}, the principal frame field $e$ gives
rise to a Laguerre frame field
$A = A(n,e_2,e_3 ;x)$ along the Legendre lift $f=(x,n)$ of the immersion $x$. An easy
computation shows that the
frame field $A$ satisfies the following equations
\[
\begin{cases}
\displaystyle dA_0 = \varphi^2 A_2 +\varphi^3 A_3,\\
\displaystyle dA_1 = \frac{\varphi^2_1}{\sqrt2} A_2 +\frac{\varphi^3_1}{\sqrt2} A_3,\\
\displaystyle dA_2 = -\frac{\varphi^2_1}{\sqrt2} A_1 +\varphi^3_2 A_3 + \frac{\varphi^2_1}{\sqrt2} A_4 +\varphi^2 A_5, \\
\displaystyle dA_3 = -\frac{\varphi^3_1}{\sqrt2} A_1 -\alpha^3_2 A_2 + \varphi^3_1 A_4 +\varphi^3 A_5,\\
\displaystyle dA_4 = -\frac{\varphi^2_1}{\sqrt2} A_2 - \frac{\varphi^3_1}{\sqrt2} A_3, \\
dA_5 =0.
\end{cases}
\]
According to \eqref{dA=Aalpha}, we have
\[
 \alpha^2_0 = -\frac{\sqrt2}{a} \varphi^2_1, \quad \alpha^3_0 = -\frac{\sqrt2}{c} \varphi^3_1.
  \]
Following the reduction procedure described
in \S \ref{ss:construction}, we may adapt $A$ to a first order Laguerre frame
field $A'$ along the Legendre lift $f$, by setting
\[
 A' = AX(d;I_2;0),  \quad d = \big(\sqrt2\frac{H}{K},1\big),
  \]
where $H =\frac{1}{2}(a +c)$ is the mean curvature and $K =ac$ the Gauss curvature of $x$.
The 1-form ${\alpha'}^2_0$ and ${\alpha'}^3_0$ take the form
\[
 {\alpha'}^2_0 = -\frac{c-a}{2K} \varphi^2_1, \quad {\alpha'}^3_0 = \frac{c-a}{2K} \varphi^3_1.
  \]
  Therefore,
\[
 \Phi_f = \frac{H^2-K}{K^2} dn \cdot dn, \quad \Omega_f =-\frac{H^2-K}{K} dA,
  \]
where $dA$ denotes the Euclidean area element of $x: S \to \mathbb R^3$.

\begin{remark}
An easy computation shows that, with respect to the first order frame field $A'$, the Laguerre Gauss map
$\sigma_f : S \to \mathcal Q_\Sigma$ is expressed by
\[
 \sigma_f = S_{\frac{H}{K}}\Big(x + \frac{H}{K} n\Big),
\]
that is, $\sigma_f(s)$ represents the oriented sphere centered at $x(s) + \frac{H}{K}(s) n(s)$ with signed radius
$\frac{H}{K}(s)$.
\end{remark}

If we let
\[
d\Big( \frac{H}{K}\Big) = \frac{c-a}{2K}\left(\Big( \frac{H}{K}\Big)_2\varphi^2_1+
\Big( \frac{H}{K}\Big)_3 \varphi^3_1 \right),
\]
we compute
\[
 {\alpha'}^1_0 = -\sqrt2\Big( \frac{H}{K}\Big)_2{\alpha'}^2_0+
  \sqrt2\Big( \frac{H}{K}\Big)_3{\alpha'}^3_0.
    \]
    We then further adapt $A'$ to a second order frame field $A''$ along $f$,
    \[
     A'' = A' X(d;I_2;y), \quad d= (0,1), \quad y = {}^t\!\Big(-\sqrt2\Big( \frac{H}{K}\Big)_2,
     \sqrt2\Big( \frac{H}{K}\Big)_3 \Big).
    \]
It turns out that the frame field $A''$ is actually a canonical frame field along $f$ and that the
corresponding canonical coframe field is given by
\[
 {\alpha''}^2_0 = -\frac{c-a}{2K} \varphi^2_1, \quad {\alpha''}^3_0 = \frac{c-a}{2K} \varphi^3_1.
  \]

\section{$L$-isothermic surfaces}\label{s:L-isotermic}
An important class of surfaces that belong to Laguerre geometry is that
of $L$-isothermic surfaces.

\begin{defn}
A nondegenerate Legendre immersion $f : S \to \La$, with canonical frame field $A$,
is called {\it $L$-isothermic}
if there exist local coordinates which simultaneously diagonalize
the definite pair of quadratic forms
$\langle dA_0,dA_0\rangle =(\alpha^2_0)^2+(\alpha^3_0)^2$ and
$\langle dA_0,dA_1\rangle =(\alpha^2_0)^2-(\alpha^3_0)^2$
and which are
isothermal with respect to $\langle dA_0,dA_0\rangle$.
\end{defn}
\begin{remark}
If $x : S \to \R^3$ is an immersed surface
without umbilic and parabolic points, oriented by the unit normal field $n$,
the $L$-isothermic condition amounts to the existence of
isothermal (conformal)
curvature line coordinates for the pair of quadratic forms $\mathrm{III}=dn \cdot dn$
and $\mathrm{II}=df\cdot dn$. Examples of $L$-isothermic surfaces include surfaces of revolution,
molding surfaces,
surfaces with plane lines of curvature \cite{MN-AMB}, minimal surfaces, etc.
The notion of $L$-isothermic surfaces is already in Bianchi \cite{Bi}, but
the systematic study of $L$-isothermic surfaces was initiated and developed by Blaschke \cite{Blaschke}.
\end{remark}

The following result was proved in \cite{MN-BOLL}.

\begin{thm}[\cite{MN-BOLL}]
A nondegenerate Legendre immersion $f : S \to \La$ is $L$-isothermic if and only if $p_2 =0$.
\end{thm}

We first came across $L$-isothermic surfaces when addressing the question of applicability
of surfaces in Laguerre geometry. We recall the following.

\begin{defn}\label{d:applicable}
Let $f,\hat f : S \to \La$ be two nondegenerate Legendre immersions, with
respective canonical frame fields
$A$, $\hat A$ and canonical coframe fields $(\alpha^2_0, \alpha^3_0)$,
$(\hat\alpha^2_0, \hat\alpha^3_0)$. If $f$ and $\hat f$ are not $L$-equivalent, then
$f$ and $\hat f$ are {\it $L$-applicable}
on each other if $\alpha^2_0 = \pm \hat\alpha^2_0$, $\alpha^3_0 = \pm \hat\alpha^3_0$.
We say that $f: S \to \La$ is $L$-applicable if, for each $s\in S$, there exists an
open neighborhood $U$ of $s$, such that $f_{|U}$ is $L$-applicable on some $\hat f$.
\end{defn}

\begin{remark}
If $f =(x,n)$ and $\hat f = (\hat x, \hat n)$ are the Legendre lifts of two immersions
$x : S\to \R^3$ and $\hat x : S\to \R^3$, with unit normal fields $n$ and $\hat n$, respectively,
then $f$ and $\hat f$ are $L$-applicable if there exists a (local) orientation-preserving conformal
transformation between $(S,\,\mathrm{III}=dn\cdot dn)$ and $(S,\, \hat{\mathrm{III}}=d\hat n\cdot d\hat n)$
that preserves the curvature lines
of $x$ and $\hat x$.
\end{remark}

\begin{remark}\label{rk:def-pbm}
The above definition is related to the general deformation theory of submanifolds in
homogeneous spaces as developed by Fubini, Cartan, Griffiths and Jensen (cf. \cite{Ca1, Gr, J}).
In fact, it has been proved that two Legendre immersions $f$ and
$\hat f$ are $L$-applicable if and only if there exists a smooth map $B : S \to L$ with the property
that $B(s) \hat f$ and $f$ agree up to order two at $s$, for each $s\in S$. In other words, if
and only if $f$ and $\hat f$
are second order deformations of each other with respect to the Laguerre group.
For the proof of this result and for more details on the deformation problem of
surfaces in Laguerre geometry we refer to \cite{MN-BOLL} and \cite{MN-TOH}.
\end{remark}

The notion of $L$-applicability introduced in Definition \ref{d:applicable} is the analogue, in Laguerre geometry,
of the following
notion of applicability in Euclidean geometry: two noncongruent (relative to rigid motions) immersions
$x, \hat x: S \to \R^3$ are
applicable if they are connected by a (local) isometry that preserves the lines of curvature. It is known that
{\it the only applicable surfaces in this sense
are the molding surfaces}.\footnote{We recall (cf. \cite{BCG,BCGGG}) that a molding surface $S$ can
be described kinematically as follows: Take a cylinder $Z$ and a curve $C$ on one of the tangent
planes to $Z$. The surface $S$ is the locus described by $C$ as the plane rolls about $Z$.
Among the molding surfaces there are the surfaces of revolution.}
{\it Actually, such surfaces belong to a 1-parameter family of noncongruent surfaces, which are connected
by isometries preserving the
lines of curvature}. For more on molding surfaces and for the proof of this result we refer
to \cite[\S5, Theorem 5.1]{BCG} or \cite[Chapter IV, Theorem 8.1]{BCGGG}.
The above notion of $L$-applicability is also related to the
notion of applicability in M\"obius geometry: two noncongruent (relative to M\"obius transformations) immersions
$x, \hat x: S \to \R^3\cup\{\infty\}$ are
$M$-applicable if they are connected by a (local) conformal transformation that preserves the lines of curvature.
{\it The only $M$-applicable surfaces
are the isothermic surfaces}.\footnote{We recall that a surface is
isothermic if it admits conformal curvature line coordinates away from umbilic points.}
{\it Also in this case, such surfaces belong to a 1-parameter family of noncongruent surfaces, which are connected
by conformal transformations preserving the
lines of curvature} (cf. \cite{Mu-TS}). For more on isothermic surfaces in M\"obius geometry, we refer
to \cite{HJlibro} and \cite{JMN}.

\vskip0.1cm
$L$-applicable surfaces can be characterized in a similar way.

\begin{thm}[\cite{MN-BOLL}]\label{thm:L-applicable}
The $L$-isothermic immersions are the only Legendre immersions that are $L$-applicable.
Moreover, any $L$-isothermic immersion belongs to a 1-parameter family of non-equivalent $L$-applicable
Legendre immersions.
\end{thm}

For a given $L$-isothermic immersion, we will now describe the nontrivial 1-parameter family
of its $L$-applicable immersions.
If $f: S \to \La$ is $L$-isothermic, there are local
curvature line coordinates $z= x+iy$
such that the canonical coframe
$(\alpha^2_0, \alpha^3_0)$ takes the form
\[
 \alpha^2_0= e^u dx,\quad \alpha^3_0 = e^u dy,
  \]
for a smooth function $u$ on $S$.
We call $\Phi =e^u$ the {\it Blaschke potential} of $f$.

Accordingly, from \eqref{se0}, \eqref{se1} and \eqref{se2} it follows that
\begin{eqnarray}
  q_1= -e^{-u}u_y, & {} &q_2= e^{-u}u_x, \label{iso1} \\
   p_1 -p_3 &=& - e^{-2u} \Delta u \label{iso2}.
    \end{eqnarray}
Moreover, using \eqref{se3} and \eqref{se4} yields
\begin{equation}\label{dW}
 \aligned
  d\left(e^{2u}(p_1 +p_3)\right) &= -e^{2u}\left\{\left(e^{-2u}\Delta u\right)_x
   +4u_x(e^{-2u}\Delta u)\right\} dx \\
   &\quad +
    e^{2u}\left\{(e^{-2u}\Delta u)_y
     +4u_y(e^{-2u}\Delta u)\right\}dy.
  \endaligned
     \end{equation}
The integrability condition of \eqref{dW} is the so-called {\it Blaschke equation},
\begin{equation}\label{blaschke-eq}
 \Delta \left(e^{-u}(e^u)_{xy}\right) =0,
  \end{equation}
which can be viewed as the completely integrable (soliton)
equation governing $L$-iso\-thermic surfaces (cf. \cite{MN-BUD, MN-AMB}).

Conversely, let $U$ be a simply connected domain in $\C$, and let $\Phi = e^u$ be
a solution to the Blaschke equation \eqref{blaschke-eq}. It follows
that the right hand side of \eqref{dW} is a closed 1-form, say $\eta_\Phi$.
Thus, $\eta_\Phi = d\mathcal K$, for some function $\mathcal K$ determined up to an
additive constant.
If we let
\begin{equation}\label{J-L}
    \text{\sc w} = \mathcal Ke^{-2u}, \quad
   \text{\sc j} = -\frac{1}{2}e^{-2u}\Delta u
   \end{equation}
(cf. \eqref{mixed-invariants} for the definition of $ \text{\sc w}$ and $\text{\sc j}$), the 1-form defined by
\[
\alpha =
 \begin{bmatrix}
0 & 0 &  0  &  0  &  0  & 0 \\
0 &2du&(\text{\sc w}+\text{\sc j})e^udx&(\text{\sc w}-\text{\sc j})e^udy&0 & 0 \\
e^udx  &e^udx& 0 &u_ydx - u_xdy &(\text{\sc w}+\text{\sc j})e^udx & 0 \\
e^udy  &-e^udy&-u_ydx +u_xdy&0&(\text{\sc w}-\text{\sc j})e^udy & 0\\
0 & 0&e^udx&-e^udy&-2du & 0 \\
0  &  0 & e^udx &  e^udy & 0 & 0
\end{bmatrix}
\]
satisfies the Maurer--Cartan integrability condition
$$d\alpha +\alpha\wedge\alpha = 0$$
and then integrates to a
map $A=  U \to L$, such that $dA = A\alpha$.
The map $f : U \to \La$ defined by
$$f= \lambda(A_0,A_1)$$
is a smooth Legendre immersion and $A$
is a canonical frame field along $f$. Thus, $f$ is an $L$-isothermic
immersion (unique up to Laguerre equivalence) and $\Phi$ is its Blaschke potential.

If, for any $m\in \R$, we let
\begin{equation}\label{J-L-m}
    \text{\sc w}_m = \text{\sc w}+me^{-2u}, \quad
   \text{\sc j}_m = \text{\sc j}  = -\frac{1}{2}e^{-2u}\Delta u,
   \end{equation}
then the 1-form defined by
\[
\alpha^{(m)}=
 \begin{bmatrix}
0 & 0 &  0  &  0  &  0  & 0 \\
0 &2du&(\text{\sc w}_m+\text{\sc j}_m)e^udx&(\text{\sc w}_m-\text{\sc j}_m)e^udy&0 & 0 \\
e^udx  &e^udx& 0 &u_ydx - u_xdy &(\text{\sc w}_m+\text{\sc j}_m)e^udx & 0 \\
e^udy  &-e^udy&-u_ydx +u_xdy&0&(\text{\sc w}_m-\text{\sc j}_m)e^udy & 0\\
0 & 0&e^udx&-e^udy&-2du & 0 \\
0  &  0 & e^udx &  e^udy & 0 & 0
\end{bmatrix}
\]
satisfies the Maurer--Cartan integrability condition
$$d\alpha^{(m)} +\alpha^{(m)}\wedge\alpha^{(m)} = 0,$$
so that there exists a
smooth map $A^{(m)} : U \to L$, such that $dA^{(m)} = A^{(m)}\alpha^{(m)}$.
The map $f_m : U \to \La$, given by $f_m= \lambda(A_0^{(m)},A_1^{(m)})$,
is a smooth Legendre immersion and $A^{(m)}$
is a canonical frame field along $f_m$. Thus, $f_m$ is an $L$-isothermic
immersion (unique up to
Laguerre equivalence) with the same Blaschke potential $\Phi$ as $f =f_0$.
Then there exist a 1-parameter family of non-equivalent $L$-isothermic immersions $\{f_m\}$
with the same Blaschke potential $\Phi$. This family amounts to the 1-parameter family of Legendre
immersions that are
$L$-applicable on $f$ (cf. Theorem \ref{thm:L-applicable}).

Actually, any other nondegenerate $L$-isothermic immersion having $\Phi$
as Blaschke potential is Laguerre equivalent to $f_m$, for some $m\in \R$.

\begin{defn}
Two $L$-isothermic immersions
$f, \tilde f$ which are not Laguerre equivalent are said to
be $T$-{\it transforms} ({\it spectral deformations}) of each other
if they have the same Blaschke potential.
\end{defn}

\begin{remark}
The spectral family $f_m$ describes all $T$-transforms
of $f= f_0$. In fact, any nondegenerate $T$-transform of $f$ is Laguerre
equivalent to $f_m$, for some $m\in \R$.
Such a 1-parameter family of $L$-isothermic surfaces amounts to the family of
second order Laguerre deformations of $f$ in the sense of Remark \ref{rk:def-pbm}.
\end{remark}

\section{$L$-minimal surfaces}\label{s:L-minimal}

Let $f : S \to \La$ be a Legendre surface. A
compactly supported variation of $f$ is a differentiable mapping
\[
V : S \times (-\epsilon,\epsilon) \to \La, \, (s,t) \longmapsto V(s,t), \quad \text{for some}\,\,\epsilon > 0,
    \]
such that its restriction $f_t$ to $S \times\{t\}$, $t \in(-\epsilon,\epsilon)$, is a Legendre surface, $f_0 = f$, and
such that there exists a compact domain $K \subset S$ for which $V (s,t) = f(s)$, for every
$s \in S \setminus K$ and every $t \in(-\epsilon,\epsilon)$. If $f$ is nondegenerate we may suppose that $f_t$ is also
nondegenerate, for each $t \in(-\epsilon,\epsilon)$.
Given a compact domain $C \subset S$, we define the functional
\begin{equation}\label{L-functional}
  \mathcal W_C (f) =\int_C \Omega_f
     \end{equation}
on the space of smooth Legendre immersions $f : S \to \La$. We call \eqref{L-functional} the Weingarten functional.
A Legendre immersion $f  : S \to \La$ is called $L$-minimal if it is
an extremal of \eqref{L-functional}, that is, if for any compact domain $C \subset S$ and any
differentiable variation $f_t$ with support in $C$ we have
\[
\frac{d}{dt} (\mathcal W_C (f_t)) \Big|_{t=0} = 0.
   \]

The following result holds true.

\begin{thm}[\cite{MN-TAMS}]
A nondegenerate Legendre immersion $f :S \to \La$ is $L$-minimal if and only if $p_1 +p_3 =0$.
\end{thm}

\begin{remark}
If $f =(x, n) : S\to \La$ is the Legendre lift of an immersed surface $x : S \to \R^3$, oriented
by a unit normal field
$n : S \to S^2$, the functional \eqref{L-functional} coincides with the Weingarten functional
\[
 \mathcal W(S,x) = \int \frac{H^2 -K}{K} dA,
  \]
where $H$ and $K$ are the
mean and Gauss curvatures of $x$ and $dA$ is the induced area element of the surface.
In this case,
using the computations made in \S \ref{ss:euclidean-inv}, one can write the invariant functions
$p_1$ and $p_3$ in terms of the Euclidean invariants and
show (cf. e.g. \cite{MN-TAMS}) that the condition to be $L$-minimal is expressed by
the nonlinear fourth-order elliptic PDE
\begin{equation}\label{euler-lagrange-eq}
 \Delta^{\mathrm{III}}\left(\frac{H}{K} \right) = 0,
  \end{equation}
where $\Delta^{\mathrm{III}}$ denotes the Laplace--Beltrami operator with respect to the third fundamental form
$\mathrm{III}$ of the immersion $x$.
\end{remark}

Nondegenerate $L$-minimal surfaces are characterized by the
minimality
of the Laguerre Gauss map.

\begin{thm}[\cite{Blaschke}, \cite{MN-TAMS}]\label{thm:l-minimal-minimal}
A nondegenerate Legendre immersion $f :S \to \La$ is $L$-minimal if and only if its Laguerre Gauss map
$\sigma_f :S \to \mathcal Q_\Sigma \cong \R^{3,1}$ has zero mean curvature vector.
\end{thm}

\begin{proof}
According to Remark \ref{rk:sigma-f-spacelike},
the Laguerre Gauss map $\sigma_f: S \to \R^{3,1}$ is a spacelike immersion
with induced metric $\Phi_f$
and relative induced area element $\Omega_f$. In particular,
\[
 \mathcal W_K (f) =\int_K \Omega_f = \text{Area}(\sigma_f).
\]
From this it is clear that if $\sigma_f$ has zero mean curvature vector, then $f$ must be $L$-minimal.
For the necessity condition, we need to compute the mean curvature vector of $\sigma_f$.
The canonical frame field $A$ along $f$ is adapted to the Laguerre Gauss map
$\sigma_f$. In fact,
the bundle induced by $\sigma_f$ over $S$ splits into the direct sum
\[
 \sigma_f^\ast (T \R^{3,1}) = T(\sigma_f) \oplus N(\sigma_f),
  \]
where $T(\sigma_f) =\span\{A_2, A_3\}$ is the tangent bundle
and $N(\sigma_F) =\span\{A_1, A_4\}$ is the normal bundle of $\sigma_f$.
The metric induced by
$\sigma_f$ on $S$, is
\[
 \Phi_f=\langle d\sigma_f, d\sigma_f\rangle
  = (\alpha^2_0)^2 +(\alpha^3_0)^2,
  \]
and $\alpha^2_0$, $\alpha^3_0$ defines an orthonormal coframe field on $S$.
As $d\sigma_f(TS) = \span\{A_2, A_3\}$,
it follows from \eqref{lframe2a} that
\[
  \alpha^1_0 = 0 = \alpha^4_0.
  \]
Differentiating these equations yields
\[
\aligned
 0 &= d\alpha^1_0 = -\alpha^1_2 \wedge\alpha^2_0 - \alpha^1_3\wedge\alpha^3_0,\\
 0 &= d\alpha^4_0 = -\alpha^4_2 \wedge\alpha^2_0 - \alpha^4_3\wedge\alpha^3_0.
\endaligned
\]
Then, by Cartan's Lemma,
\[
 \alpha^{\nu}_i = h^{\nu}_{i2}\alpha^2_0 + h^{\nu}_{i3}\alpha^3_0,
 \quad  h^{\nu}_{ij}=h^{\nu}_{ji} \quad \nu=1,4;\, i,j =2,3,
  \]
where the functions $h^{\nu}_{ij}$ are the components of the
second fundamental form of $\sigma_f$,
\[
 \Pi = \sum_{i,j=2,3} h^1_{ij}\alpha^i_0\alpha^j_0\otimes a_4
 + \sum_{i,j=2,3} h^4_{ij}\alpha^i_0\alpha^j_0 \otimes a_1.
   \]
From \eqref{alfa21-31}, \eqref{alfa12-13} and the symmetry relations \eqref{str-eq1}, it follows that
\[
  (h^1_{ij})= \begin{bmatrix} 1& 0\\0& -1 \end{bmatrix},
  \quad  (h^4_{ij})= \begin{bmatrix} p_1& p_2\\p_2&p_3 \end{bmatrix}.
   \]
The mean curvature vector $\mathbf{H}$ of $\sigma_f$ is half the
trace of $\Pi$
with respect to $\Phi_f$,
i.e.,
\begin{equation}\label{mean-curv-vector}
  2 \mathbf{H} = (p_1 + p_3) A_1.
  \end{equation}
Thus, $\mathbf{H}\equiv 0$ on $S$ if and only if
$p_1+ p_3$ vanishes identically on $S$.
\end{proof}

\begin{remark}
(1) In the presence of umbilical points, one can prove that $\sigma_f$ is a conformal harmonic map.
(2) The necessity condition in Theorem \ref{thm:l-minimal-minimal} is sort of surprising since
there are
more variations of the map $\sigma_f$ than there are variations through Laguerre Gauss maps.
(3) Note that $\mathbf{H}$ is a null section of the
normal bundle $N(\sigma_f)$, i.e., $\langle\mathbf{H}, \mathbf{H} \rangle=0$.
In particular, it follows from \eqref{mean-curv-vector} that $\sigma_f$ is a marginally outer trapped
surface (MOTS) in $\R^{3,1}$ (cf. \cite{CGP-BAMS, HE} for more details on MOTS and \cite{MNjgp}
for the analogous situation in M\"obius geometry).
(4) With respect to the null frame field $\{A_1, A_4\}$,
the normal connection $\nabla^\perp$ in the normal bundle $N(\sigma)$
of $\sigma$ is
given by
\[
 \nabla^\perp a_1 =  \alpha^1_1 \otimes a_1, \quad
  \nabla^\perp a_4 =  -\alpha^1_1 \otimes a_4.
  \]
In particular, we have
\begin{equation}\label{H-normal-deriv}
 2\nabla^\perp \mathbf{H} =\left[d(p_1 +p_3) + (p_1 +p_3) \,\alpha^1_1\right] a_1,
   \end{equation}
so that the parallel condition $\nabla^\perp \mathbf{H} = 0$ takes the form
\begin{equation}\label{gLm-eq}
 d(p_1+p_3) + 2(p_1+p_3)(q_2\alpha^2_0 -q_1 \alpha^3_0) = 0.
  \end{equation}
\end{remark}


\section{The Pfaffian system of $L$-minimal surfaces}\label{s:Pfaffian-sys}

In this section we introduce the Pfaffian differential systems of $L$-minimal surfaces and
prove that it is in involution. Then,
as an application of the general Cartan--K\"ahler theorem for exterior differential
systems in involution we study the Cauchy problem for $L$-minimal surfaces
and prove the existence of a unique real analytic $L$-minimal surface passing through
a given real analytic integral curve of the system.
The functional dependence of the initial curve amounts to the choice of four arbitrary
functions in one variable. For the study of the Cauchy problem in other geometric situations
we refer to \cite{JMN-JPA, MNjmp, MNphysD}.

As shown above, $L$-minimal surfaces are characterized by the condition $p_1 +p_3 =0$. Hence, by reasoning as
in \S \ref{ss:LPS}, they can be interpreted as integral manifolds of the differential system obtained by
restricting the Laguerre differential system $(\mathcal I,\Omega)$ to the submanifold
\[
 Y = \left\{(A,q_1,q_2, p_1, p_2, p_3)\in L\times \mathbb R^5 \,: \, p_1 +p_3 =0\right\} \subset M.
  \]

We call $(Y,\mathcal I,\Omega)$ the {\it Pfaffian system of $L$-minimal surfaces}.
We shall identify $Y$ and $L\times \mathbb R^4$ and
denote by $(q_1,q_2, p_1, p_2)$ the points of $\mathbb R^4$. In this way, the system $(\mathcal I,\Omega)$
on $Y$ is
differentially generated by the 1-forms
\[
\begin{aligned}
&\eta^1 = \omega^4_0, \quad \eta^2 = \omega^1_0, \quad \eta^3 = \omega^2_1-\omega^1, \quad \eta^4 = \omega^3_1+\omega^1,\\
&\eta^5 = \omega^3_2 -q_1\omega^1 -q_2\omega^2, \quad \eta^6 = \omega^1_1 -2q_2\omega^1 +2q_1\omega^2, \\
&\eta^7 = \omega^1_2 -p_1\omega^1 -p_2\omega^2, \quad \eta^8 = \omega^1_3 -p_2\omega^1 +p_1\omega^2,
\end{aligned}
\]
with independence condition
\[
 \omega^1\wedge \omega^2 \neq 0,
  \]
where $\omega^1 = \omega^2_0$ and $\omega^2 = \omega^3_0$. Moreover, we let $\pi^i = dq_i$, $\zeta^i = dp_i$,
$i=1,2$, be the set of 1-forms that complete $\{\eta^a, \omega^i\}$ to a global coframe field on $Y$.

\subsection{Quadratic equations and involution of the system}

From the structure equations \eqref{str-eq2} of the Laguerre group, exterior differentiation of $\eta^1, \dots \eta^8$
yields the following equations
\begin{equation}\label{qe-LMPS}
\begin{aligned}
&d\eta^1 \equiv d\eta^2 \equiv d\eta^3 \equiv d\eta^4 \equiv 0 \mod \{\eta^a\}, \\
&d\eta^5 \equiv -\pi^1 \wedge \omega^1 -\pi^2\wedge \omega^2 -(2p_1 +{q_1}^2 + {q_2}^2)\, \omega^1 \wedge \omega^2 \mod \{\eta^a\},  \\
&d\eta^6 \equiv 2\pi^1 \wedge \omega^2 -2\pi^2\wedge \omega^1 +2p_2 \, \omega^1 \wedge \omega^2 \mod \{\eta^a\},  \\
&d\eta^7 \equiv -\zeta^1 \wedge \omega^1 -\zeta^2\wedge \omega^2 -4 (q_1p_1 +q_2p_2 )\, \omega^1 \wedge \omega^2 \mod \{\eta^a\},  \\
&d\eta^8 \equiv \zeta^1 \wedge \omega^2 -\zeta^2\wedge \omega^1 -4 (q_1p_2 -q_2p_1 )\, \omega^1 \wedge \omega^2 \mod \{\eta^a\},
\end{aligned}
\end{equation}
where $\{\eta^a\}$ denotes the algebraic ideal generated by the 1-forms $\eta^1, \dots \eta^8$.
The equations in \eqref{qe-LMPS} are referred to as the {\it quadratic equations} of $(Y,\mathcal I,\Omega)$.
The reduced tableaux matrix $\mu$ of the system is then
\[
{}^t\!\mu =
\begin{bmatrix}
0 & 0 & 0 & 0 & -\pi^1 & -2\pi^2 & -\zeta^1 & -\zeta^2\\
0 & 0 & 0 & 0 & -\pi^2 &  2\pi^1 & -\zeta^2 & \zeta^1 \\
\end{bmatrix}.
\]
It then follows that the reduced Cartan characters $s_1'$, $s_2'$ are
\begin{equation}\label{reduced-chars}
  s_1' = 4, \quad s_2' = 0.
   \end{equation}
Let $V_2(\mathcal I)$ denote the set of all 2-dimensional integral elements of $(\mathcal I, \Omega)$, i.e.,
\begin{equation}\label{2-dim-int-el}
 V_2(\mathcal I) = \{ (\mathfrak z, E_2) \in G_2(TY) \,: \, {\eta^a}_{|E_2}
 = 0, \, {d\eta^a}_{|E_2} = 0, \, {\omega^1 \wedge \omega^2}_{|E_2} \neq 0\},
   \end{equation}
and let $O_2(\mathcal I)$ be the open subset of the Grassmannian $G_2(TY)$ consisting of all tangent planes
$W \subset T_{\mathfrak z}Y $ such that ${\omega^1 \wedge \omega^2}_{|W} \neq 0$.
On $O_2(\mathcal I)$, we have fiber coordinates $(X^a_i, V^j_i,W^j_i)$, $a = 1,\dots,8$, $i,j=1,2$, defined by
\[
 {\eta^a}_{|W} = X^a_i(W) {\omega^i}_{|W}, \quad
 {\pi^i}_{|W} = V^i_j(W) {\omega^j}_{|W}, \quad
  {\zeta^i}_{|W} = W^i_j(W) {\omega^j}_{|W}.
 \]
At any point $\mathfrak z = (A, q_1,q_2, p_1, p_2)$ of $Y$, the fiber $O_2(\mathcal I)_{\mathfrak z}$
is identified with the affine space $\mathbb A^{24}$ by the map
\[
 O_2(\mathcal I)_{\mathfrak z} \ni W \longmapsto (X^a_i(W), V^j_i(W),W^j_i(W)).
  \]
  From the quadratic equations \eqref{qe-LMPS} it follows that $V_2(\mathcal I)_{\mathfrak z}$ is the 4-dimensional affine
  subspace of $O_2(\mathcal I)_{\mathfrak z}$ defined by the equations
\begin{equation}\label{affine-eqs}
\begin{aligned}
& X^a_i = 0, \quad a= 1,\dots,8;\, i=1,2 \\
&V^1_2 -V^2_1 = 2p_1 +{q_1}^2 +{q_2}^2,  \\
&V^1_1 +V^2_2 = -p_2 ,  \\
&W^1_2 -W^2_1 = 4(q_1p_1 + q_2p_2),  \\
&W^1_1 +W^2_2 = 4(q_1p_2 - q_2p_1).
\end{aligned}
\end{equation}

This yields the following.

\begin{lemma}
The set $V_2(\mathcal I)$ of 2-dimensional integral elements of the Pfaffian system $(Y,\mathcal I, \Omega)$ is a real analytic 18-dimensional embedded submanifold of the Grassmannian $G_2(TY)$
and the fiber over $\mathfrak z \in Y$ of the bundle
map $V_2(\mathcal I) \to Y$, $(\mathfrak z, E_2)\mapsto \mathfrak z$,
is a 4-dimensional affine subspace of $G_2(T_{\mathfrak z}Y)$, for each $\mathfrak z \in Y$,
i.e.,
$$t = \mathrm{dim}\, V_2(\mathcal I)_{\mathfrak z}=4, \quad
\text{for each} \quad \mathfrak z \in Y.$$
\end{lemma}

From this and \eqref{reduced-chars}, it follows that $t = s_1' + 2s_2'$,
so that Cartan's test of involution applies. We can then state the following.

\begin{prop}
The Pfaffian differential system $(Y,\mathcal I, \Omega)$ of $L$-minimal surfaces is in
involution\footnote{that is, at every point $\mathfrak z \in Y$ there exists an {\it ordinary} integral element
$(\mathfrak z, E_2) \in V_2(\mathcal I)_{\mathfrak z}$ (cf. \cite{BCGGG} for more details).}
and its integral manifolds depend on four functions in one variable.
\end{prop}

\subsection{Polar equations and the Cauchy problem}

Next, we study the polar equations of the system $(Y,\mathcal I, \Omega)$. On the tangent bundle $TY$,
we consider fibre coordinates $(a^i, X^a, V^i,W^i)$ defined, for any tangent vector $\xi$, by
\begin{equation}\label{fiber-coord-TY}
 a^i(\xi) = \omega^i(\xi), \quad X^a(\xi) = \eta^a(\xi), \quad V^i(\xi) = \pi^i(\xi), \quad
  W^i(\xi) = \zeta^i(\xi).
   \end{equation}
Thus, $(a^i, X^a, V^i,W^i)$ can be used as homogeneous fiber coordinates on the Grassmannian $G_1(TY)$.

A 1-dimensional integral element of $(\mathcal I, \Omega)$, at a fixed point $\mathfrak z \in Y$, is
a 1-dimensional subspace $E_1 = \text{span}\{\xi\}$ of $T_{\mathfrak z}Y$, spanned by a nonzero tangent vector
$\xi$, such that
\begin{equation}\label{eqs-V1}
 \eta^a(\xi) =0, \, a=1,\dots,8; \quad \omega^1(\xi)\omega^1(\xi) + \omega^2(\xi)\omega^2(\xi) \neq 0.
\end{equation}
Let $V_1(\mathcal I)$
denote the set
of all 1-dimensional integral elements
of $(\mathcal I, \Omega)$.
%
The set $V_1(\mathcal I)$
is the submanifold of $G_1(TY)$
defined by the equations
\begin{equation}\label{1-integral}
  X^a = 0, \quad a= 1,\dots,8;\quad (a^1)^2 + (a^2)^2  \neq 0.
   \end{equation}

   Let $(\mathfrak z, E_1)$ be a 1-dimensional integral element. We recall that the {\it polar} or
   {\it extension space}
   $H(\mathfrak z, E_1)$ of $(\mathfrak z, E_1)$ is the subspace of $T_{\mathfrak z}Y$ defined
   by the {\it polar equations}
\begin{equation}\label{polar-space-eqs}
  \eta^a = 0, \quad \imath_\xi d\eta^a = 0, \quad a= 1,\dots,8.
   \end{equation}

Let $\mathfrak z = (A, q_1,q_2, p_1, p_2)\in Y$ and $(\mathfrak z, E_1)$ a
1-dimensional integral element with $E_1 = \text{span}\,\{\xi\}$,
$\xi \neq 0$.
From \eqref{polar-space-eqs}, \eqref{1-integral}, \eqref{fiber-coord-TY}, using the quadratic equations
\eqref{qe-LMPS} of the system, the polar equations read
%
%
\begin{eqnarray*}
&& \eta^a =0, \quad a= 1,\dots,8,\\
&&a^1 \pi^1 + a^2 \pi^2 + [a^2(2p_1 + {q_1}^2 +{q_2}^2) - V^1]\, \omega^1 + \\
     & &\hskip4cm -\, [a^1(2p_1 + {q_1}^2 +{q_2}^2) + V^2]\, \omega^2=0, \nonumber\\
&& -a^2 \pi^1 + a^1 \pi^2 - [V^2 + p_2 a^2]\, \omega^1
 + [V^1 + p_2 a^1]\, \omega^2 =0, \\
&&a^1 \zeta^1 + a^2 \zeta^2  - [W^1 + 4a^1(q_1p_1 + q_2p_2)]\, \omega^1 + \\
    & &\hskip4cm -\, [W^2 - 4a^2(q_1p_1 + q_2p_2)]\, \omega^2=0, \nonumber\\
&&-a^2 \zeta^1 + a^1 \zeta^2  - [W^2 - 4a^2(q_1p_2 - q_2p_1)]\, \omega^1 + \\
     && \hskip4cm +\, [W^1 - 4a^1(q_1p_2 - q_2p_1)]\, \omega^2=0. \nonumber
  \end{eqnarray*}

The polar equations have maximal rank and then the polar space $H(\mathfrak z, E_1)$ has dimension two,
for every $(\mathfrak z, E_1) \in V_1(\mathcal I)$.
Therefore, we are led to the following.

\begin{lemma}\label{l:K-regular}
For any $(\mathfrak z, E_1)\in V_1(\mathcal I)$, the polar space $H(\mathfrak z, E_1)$ is the unique
2-dimensional integral element of $(\mathcal I, \Omega)$ such that $E_1 \subset H(\mathfrak z, E_1)$. According to
\eqref{1-integral}, the set $V_1(\mathcal I)$ is a 19-dimensional embedded submanifold of $G_1(TY)$, and hence
any $(\mathfrak z, E_1)\in V_1(\mathcal I)$ is K\"ahler-regular and $$(0)_{\mathfrak z} \subset (\mathfrak z, E_1)
\subset H(\mathfrak z, E_1)$$ is a regular flag (cf. \cite[p. 42]{GJbook}).
\end{lemma}

By the Cartan--K\"ahler theorem (cf. \cite[p. 42]{GJbook}, or \cite[p. 81]{BCGGG}) we have the following.

\begin{lemma}\label{l:cauchy}
Let $\gamma : (-\epsilon,\epsilon) \to Y$ be a real analytic embedded curve such that $\mathrm{span}\{\dot \gamma (t)\}
\in V_1(\mathcal I)$, for each $t\in (-\epsilon,\epsilon)$. Then there exists a unique real analytic
connected integral manifold $\Sigma\subset Y$ of $(\mathcal I, \Omega)$ such that $\gamma \subset \Sigma$. The manifold $\Sigma$
is unique in the sense that any other integral manifold of $(Y,\mathcal I, \Omega)$ with these properties agrees with $\Sigma$ on an open neighborhood of $\gamma$.
\end{lemma}

This has the following geometric consequence.

\begin{thm}[The Cauchy problem]\label{thm:Cauchy-pbm}
Let $s_1, s_2, r_1, r_2 : (-\epsilon,\epsilon) \to \mathbb R$ be real analytic functions defined on an open interval, and
let $B\in L$ be an element of the Laguerre group. There exists an open neighborhood $U \subset \mathbb R^2$ of the
origin
and a unique real analytic $L$-minimal framed surface $(U,f,A)$,
satisfying the following initial
conditions
\begin{enumerate}
\item $A(0,0) = B$, \hskip0.1cm $q_i(t,0) = s_i(t)$, \hskip0.1cm $p_i(t,0) = r_i(t)$, for each
$t\in (-\epsilon,\epsilon)\cap U$,
\item ${\alpha^2_0}_{|(-\epsilon,\epsilon)\cap U} = A^*(\omega^2_0)_{|(-\epsilon,\epsilon)\cap U} = dt$, \hskip0.1cm ${\alpha^3_0}_{|(-\epsilon,\epsilon)\cap U} = A^*(\omega^3_0)_{|(-\epsilon,\epsilon)\cap U} = 0$,
\end{enumerate}
where $q_i$, $p_i$, $i=1,2$, are the invariant functions of $f$.
\end{thm}

\begin{proof}
Given the real analytic functions $s_i$ and $r_i$, consider the $\mathfrak l$-valued 1-form $\beta$ defined on $(-\epsilon,\epsilon)$
given by
\[
\beta=
 \begin{bmatrix}
 0   &  0  & 0         & 0         & 0                     & 0 \\
 0  &2s_2  & r_1       &r_2        &0                      & 0\\
 1   & 1   & 0         & -s_1      &r_1                   & 0\\
 0   & 0   &s_1       &0           &r_2                   & 0\\
 0   & 0  & 1         &  0        &-2s_2                  &0 \\
 0   & 0  & 1         &  0        &0               &0 \\
\end{bmatrix} dt.
\]
Let $\Gamma : (-\epsilon,\epsilon) \to L$ be the unique real analytic map such that
\[
 \Gamma^* (\omega) = \beta, \quad \Gamma(0) = B.
  \]
We may suppose that $\Gamma$ is an embedding, by possibly choosing a smaller $\epsilon$.
  Next, let $\gamma : (-\epsilon,\epsilon) \to Y$ be the curve defined by
  $\gamma(t) : = (\Gamma(t), s_1(t),s_2(t),r_1(t),r_2(t))$,
  for each $t$. By construction, $\gamma$ is a real analytic embedded curve such that
  $\mathrm{span}\,\{\dot \gamma (t)\} \in V_1(\mathcal I)$, for each $t\in (-\epsilon,\epsilon)$.
  According to Lemma \ref{l:cauchy}, there exists a unique real analytic embedded solution of the
  system $(Y, \mathcal I, \Omega)$, say $\Sigma = (A, q_1,q_2,p_1,p_2)  : U \to Y$, defined on an open disk $U$ containing
  $(-\epsilon,\epsilon)\times \{0\}$, such that $\Sigma(t,0) = \gamma(t)$, for each
  $t\in (-\epsilon,\epsilon)$. The framed immersion $(U, f, A)$, where $f = \pi_L\circ A$, is $L$-minimal
  and $q_1,q_2,p_1,p_2$ are the invariant functions of $f$. Thus $(U, f, A)$
  satisfies the required conditions.
\end{proof}

\begin{remark}
Note that the functional dependence of the general solutions of the system $(Y,\mathcal I, \Omega)$ agrees with that of
the initial data of Theorem \ref{thm:Cauchy-pbm}. Specifying the initial data of Theorem \ref{thm:Cauchy-pbm}
amounts to the choice of four arbitrary functions in one variable.
\end{remark}


\section{Generalized $L$-minimal surfaces}\label{s:gen-L-min}

Let $S$ be an oriented surface and let $f : S \to \La$
be a nondegenerate Legendre immersion into the Laguerre space.
Let $A : S\to L$ be a canonical frame field along $f$.
%
%
The Laguerre metric $(\alpha^2_0)^2 + (\alpha^3_0)^2$ and the area element
$\alpha^2_0\wedge\alpha^3_0$ induced by $A$ determine on $S$ an oriented conformal
structure and hence, by the existence of isothermal coordinates, a unique compatible
complex structure which makes $S$ into a Riemann surface. In terms of the canonical
frame field $A$, the complex structure is characterized by the property that
the complex-valued 1-form
\begin{equation}\label{1-0form}
 \varphi = \alpha^2_0 +i \alpha^3_0
  \end{equation}
is of type $(1,0)$.
%
Taking into account Remark \ref{r:p-q-change}, the complex-valued quartic differential form given by
\begin{equation}\label{quartic}
 \Q_f = Q  \varphi^4, \quad Q := \frac{1}{2}(p_1 - p_3) - i p_2,
  \end{equation}
and the complex-valued quadratic differential form given by
\begin{equation}\label{quadratic}
 \P_f = P \varphi^2, \quad P :=  p_1 + p_3
  \end{equation}
are globally defined on the Riemann surface $S$.

\begin{remark}
The quartic differential $\Q_f$ was considered by the authors in \cite{MN-TAMS}, where it was proved
that for $L$-minimal surfaces $\Q_f$ is holomorphic. Notice that the quadratic differential
$\P_F$ vanishes exactly for $L$-minimal surfaces.
\end{remark}

\begin{defn}
A nondegenerate Legendre immersion $f : S \to \La$ is called a {\it generalized $L$-minimal surface}
if the quartic differential $\Q_f$ is holomorphic.\footnote{The name is motivated by the name
``generalized Willmore surfaces'' (``verallgemeinerte Willmore-Fl\"achen'') used by K. Voss \cite{Voss1985} to indicate
those surfaces in conformal geometry with holomorphic Bryant
differential $(4,0)$ form \cite{Br-duality}. The name ``generalized Laguerre minimal
surfaces'' was also used with a different meaning in \cite{AGL} to indicate
those  $L$-minimal surfaces for which the Gauss curvature $K$ is allowed to vanish on a set of isolated points.}
\end{defn}

\begin{remark}[Surfaces in M\"obius geometry]\label{r:moebius-ge}
The results discussed in this section can be considered as the Laguerre geometric counterpart
of well-known results for surfaces in M\"obius geometry (cf. \cite{HJlibro, JMN} for more details).
For the sake of completeness, we briefly recall some of them.
Many features of CMC surfaces in 3-dimensional space forms,
viewed as isothermic surfaces in M\"obius space $S^3$, can be interpreted in terms of the transformation
theory of isothermic surfaces.
More specifically, it can be proved that CMC surfaces in space forms arise in associated 1-parameter families
as $T$-transforms of minimal surfaces in space forms \cite{Bianchi1905-12,Ca1903,Ca1915,HJlibro,JMN}.
Minimal surfaces in space forms are isothermic and
Willmore,
that is, are critical points of the Willmore energy
$\int(H^2- K) dA$
(cf. \cite{Blaschke,Br-duality,Thomsen}).
By a classical result of Thomsen \cite{HJlibro, JMN, Thomsen}, a Willmore surface
without umbilics
is isothermic if and only if
it is locally M\"obius equivalent to a minimal surface in
some space form.
K. Voss obtained a uniform M\"obius geometric characterization of
Willmore  surfaces and CMC surfaces in space forms
using the differential $(4,0)$ form $\Q$
introduced by
Bryant \cite{Br-duality} for Willmore surfaces
(cf. also \cite{Bo-Pe2009,Bo2012,JMN, Voss1985}).
Voss observed that $\Q$,
which indeed
may be defined for any conformal immersion of a Riemann surface $M$
into $S^3$,
is holomorphic if and only if, locally and away from umbilics and
isolated points, the immersion is Willmore or has constant mean curvature in some
space form embedded in $S^3$.
\end{remark}

It quite easy to prove the following facts. For the proof we refer to \cite{MN-MZ}.

\begin{lemma}\label{prop:hol-cond}
Let $f: S \to \La$ be a nondegenerate Legendre immersion.
Then:

\begin{enumerate}

\item The quartic differential $\Q_f$ is holomorphic if and only if
\begin{equation}\label{cns-holQ}
 dQ \wedge \varphi = -4(q_2\alpha^2_0 - q_1\alpha^3_0) Q \wedge \varphi.
  \end{equation}

\item $\Q_f$ is holomorphic
if and only if the Laguerre Gauss map of $f$ has parallel mean curvature
vector.

\item If $\Q_f$ is holomorphic, then
$\P_f$ is holomorphic.

\item If $\Q_f$ is holomorphic and $\P_f$ is non-zero, then $f$
is $L$-isothermic. 

\item If $\Q_f$ is holomorphic and $\P_f \neq 0$, then
$\Q_f = c\, \P_f^2$, for $c\in \mathbb R$.

\end{enumerate}
\end{lemma}

We are now ready to prove our next result.

\begin{prop}\label{thm:hQ-iff-W-or-special}

A nondegenerate Legendre immersion
$f : S \to \Lambda$ is generalized $L$-minimal if and only if
it is $L$-minimal,
in which case the quadratic differential $\P_F$ vanishes on $S$,
or it is $L$-isothermic
with Blaschke potential
$\Phi =e^u$
satisfying the second order partial differential equation
\begin{equation}\label{special-eq}
 \Delta u  = ce^{-2u}, 
  \end{equation}
where $c$ is a real constant.
\end{prop}

\begin{proof}
If $\Q_f$ is holomorphic and the holomorphic quartic differential $\P_f$
vanishes, then $f$ is $L$-minimal.
If instead $\P_f$ is nowhere vanishing, then
$f$ is $L$-isothermic by Lemma \ref{prop:hol-cond}.
Let $z = x+iy : U\subset S  \to \C$ be an isothermic
chart, so
that the canonical coframing
$(\alpha^2_0, \alpha^3_0)$ takes the form
$\alpha^2_0= e^udx$ and $\alpha^3_0 = e^udy$,
where $\Phi =e^u$ is the Blaschke potential  (cf. Section \ref{s:L-isotermic}).

 From \eqref{iso1} and \eqref{iso2} we get
\[
 \Q_f= \frac{1}{2}(p_1 - p_3) \omega^4 =  \text{\sc j}e^{4u} (dz)^4
=-\frac{1}{2}(e^{-2u} \Delta u) e^{4u}(dz)^4.
   \]
Since $\Q_f$ is holomorphic,
\[
(e^{-2u} \Delta u) e^{4u} = c,
  \]
for a constant $c\in \R$, that is
\[
 \Delta u  = ce^{-2u}.
  \]

Conversely, if we assume that $\Phi = e^{u}$ satisfy the equation
\eqref{special-eq}, then
the right hand side of \eqref{dW} vanishes identically, which implies
that $p_1 + p_3 = k e^{-2u}$, for a constant $k\in \R$.
A direct computation shows that $p_1 + p_3 = k e^{-2u}$ satisfies
the equation
\[
 d(p_1+p_3) + 2(p_1+p_3)(q_2\alpha^2_0 -q_1 \alpha^3_0) = 0.
	\]
This expresses the fact that the $L$-Gauss map
of $f$, $\sigma_f = [A_0]$, has
parallel mean curvature vector, or equivalently,
that the quartic differential $\Q_f$ is holomorphic.
\end{proof}

\subsection{Special $L$-isothermic surfaces and $T$-transforms}\label{ss:s-L-iso}

\begin{defn}
A nondegenerate $L$-isothermic immersion $f : S \to \La$
is called {\it special} if its Blaschke potential $\Phi = e^u$
satisfies the second order PDE \eqref{special-eq}, i.e.,
\[
 \Delta u  = ce^{-2u}, \quad c\in \R.
  \]
The constant $c$ is called the {\it character} of the special $L$-isothermic
surface $f$.
\end{defn}

\begin{ex}[$L$-minimal isothermic surfaces]
Since $L$-minimal surfaces are characterized
by the condition $p_1 +p_3 = 0$, if a
nondegenerate $L$-isothermic immersion $f :S \to \La$ is also $L$-minimal,
i.e., $p_2 =0$, then
the right hand side of \eqref{dW} is identically zero. This implies
that
$d\left(e^{2u} \Delta u\right)= 0$,
and hence the following.

\begin{prop}\label{prop:L-min-iso-are-special}
 Any nondegenerate $L$-mi\-ni\-mal
isothermic immersion $f : S \to \La$ is special $L$-isothermic.
\end{prop}

Other examples of $L$-minimal isothermic surfaces include
$L$-minimal canal surfaces \cite{MN-REND-RM, MN-AMB}.

\end{ex}


Let $f : S \to \La$ be a special $L$-isothermic immersion. From
the proof of
Proposition \ref{thm:hQ-iff-W-or-special}, the invariants $\text{\sc j}$ and
$\text{\sc w}$ of $f$ are given by
\begin{equation}\label{special-L-J}
   \text{\sc w} = ke^{-2u}, \quad
   \text{\sc j} = -\frac{1}{2}e^{-2u}\Delta u,
   \end{equation}
where $k$ is a real constant.

\begin{defn}
The constant $k$ will be referred to as the
{\it deformation} (or {\it spectral}) {\it parameter} of the special
$L$-isothermic immersion $F$.
\end{defn}

We have the following.

\begin{prop}\label{prop:special-as-Ttrans}
Any special $L$-isothermic immersion is the $T$-transform of
an $L$-minimal isothermic immersion.
\end{prop}

\begin{proof}
According to Section \ref{s:L-isotermic}, there exists, up to Laguerre equivalence,
a unique $L$-isothermic immersion $f$ with
Blaschke potential $\Phi = e^u$ satisfying \eqref{special-eq}
and with invariant functions
\[
   \text{\sc w} = 0, \quad
   \text{\sc j} = -\frac{1}{2}e^{-2u}\Delta u = -\frac{c}{2}e^{-4u}.
    \]
Since $\text{\sc w} = 0$, we have that $f$ is $L$-minimal.
Next, let $\tilde f$ be a special $L$-isothermic immersion with the same Blaschke
potential $\Phi$ as $f$
and with deformation parameter $k$. The discussion in Section \ref{s:L-isotermic}
implies that $\tilde f$ is a $T_m$-transform of $f$. The invariants of $\tilde f$
are then given by
\begin{equation}\label{special-L-J-m}
   \text{\sc w}_m = me^{-2u}, \quad
   \text{\sc j}_m = \text{\sc j}= -\frac{c}{2}e^{-4u}.
  \end{equation}
From \eqref{special-L-J} and \eqref{special-L-J-m}, it follows that $m=k$.
\end{proof}

From Proposition \ref{thm:hQ-iff-W-or-special},
Proposition \ref{prop:L-min-iso-are-special}, and
Proposition \ref{prop:special-as-Ttrans}, we have the following.

\vskip0.3cm

\begin{thm}[\cite{MN-MZ}]
A nondegenerate Legendre immersion $f : M \to  \La$ is generalized $L$-minimal
if and only if the immersion $f$ is $L$-minimal, in which case $\P_f$ vanishes,
or is locally the $T$-transform of an $L$-minimal isothermic surface.

\end{thm}

\vskip0.3cm
In particular, if $f$ has holomorphic $\Q_f$ and zero $\P_f$, then
$f$ is $L$-isothermic if and only if it is $L$-minimal isothermic.

$L$-minimal isothermic surfaces
and their $T$-transforms (i.e., special $L$-isothermic
surfaces with non-zero deformation parameter) can be characterized in terms of the differential geometry
of their Laguerre Gauss maps. The following result was proved in \cite{MN-MZ}.

\begin{thm}[\cite{MN-MZ}]
Let $f : S \to  \La$ be a nondegenerate Legendre immersion.
Then:

\begin{enumerate}

\item
$f$ is $L$-minimal and $L$-isothermic
if and only if
its Laguerre Gauss map $\sigma_f : S \to \R^{3,1}\cong \mathcal Q_\Sigma$
has zero mean curvature
in some
spacelike, timelike, or (degenerate) isotropic hyperplane of $\R^{3,1}$.

\item $f$ is generalized $L$-minimal
and non-zero $\P_F$
if and only if
its Laguerre Gauss map $\sigma_f : S \to \R^{3,1}$
has constant mean curvature $H=r$
in some translate of hyperbolic 3-space $\mathbb H^3(-r^2)\subset \R^{3,1}$,
de Sitter 3-space $\mathbb S^3_1(r^2)\subset \R^{3,1}$,
or has zero mean curvature in some translate of a time-oriented lightcone
$\mathcal L^3_\pm\subset \R^{3,1}$.

\end{enumerate}

In addition, if the Laguerre Gauss map of $f$
takes values in a spacelike (respectively, timelike, isotropic) hyperplane, then the
Laguerre Gauss maps of the $T$-transforms of $f$ take values
in a translate of a hyperbolic 3-space (respectively, de Sitter 3-space,
time-oriented lightcone).
In particular, the signature of the 3-space where the Laguerre Gauss map takes
values remains unchanged under $T$-transformations.
\end{thm}

\begin{remark}
As an application of these results, one can show (cf. \cite{MN-MZ}) that
the Lawson correspondence \cite{Lawson} between
certain isometric CMC surfaces in different hyperbolic 3-spaces
and, in particular,
the Umehara--Yamada isometric perturbation \cite{UY-Crelle}
of minimal surfaces of $\R^3$ into CMC surfaces in hyperbolic 3-space,
can be viewed as a special case of the $T$-transformation of $L$-isothermic surfaces
with holomorphic quartic differential.
(For a M\"obius geometric interpretation of the Umehara--Yamada
perturbation see \cite{HMN, MN-HOUSTON}).
This interpretation also applies
to the generalizations of Lawson's correspondence in the Lorentzian
\cite{Pa1990} and the (degenerate) isotropic situations,
namely to the perturbation of maximal surfaces in Minkowski
3-space into CMC spacelike surfaces in de Sitter 3-space
\cite{AiAk, AGM, Kob, Lee2005}, and that of zero mean curvature
spacelike surfaces in a (degenerate) isotropic 3-space into zero mean curvature spacelike
surfaces in a time-oriented lightcone of $\R^{3,1}$.
\end{remark}


\bibliographystyle{amsalpha}

\begin{thebibliography}{AA}

\bibitem{AiAk}
R. Aiyama and K. Akutagawa,
\emph{Kenmotsu-Bryant type representation formulas for constant mean  curvature surfaces
in $H^3(-c^2)$ and $S^3_1(c^2)$},
{Ann. Global Anal. Geom.} {17} (1999), no. 1, 49--75.

\bibitem{AGM}
J. A. Aledo, J. A. G\'alvez, and P. Mira,
\emph{Marginally trapped surfaces in $\mathbb L^4$ and an extended
Weierstrass-Bryant representation},
{Ann. Global Anal. Geom.} {28} (2005), no. 4, 395--415.

\bibitem{AGL}
J. A. Aledo, J. A. G\'alvez, and V. Lozano,
\emph{Complete Laguerre minimal surfaces in $\mathbb R^3$},
{Nonlinear Anal.} {92} (2013), 1--12.

\bibitem{Bianchi1905-12}
L. Bianchi,
\emph{Complementi alle ricerche sulle superficie isoterme},
{Ann. Mat. Pura Appl.} {12} (1905), 19--54.

\bibitem{Bi}
{L. Bianchi},
{\em Lezioni di geometria diferenziale},
terza edizione, Zanichelli, Bologna, 1927.

\bibitem{Blaschke1}
W. Blaschke,
\emph{\"Uber die Geometrie von Laguerre:} I,
{Abh. Math. Sem. Univ. Hamburg} {3} (1924),
176--194;
II,
{Abh. Math. Sem. Univ. Hamburg} {3} (1924),
195--212;
III,
{Abh. Math. Sem. Univ. Hamburg} {4} (1925)
1--12.

\bibitem{Blaschke}
{W. Blaschke},
{\em Vorlesungen \"uber Differentialgeometrie. III:
Differentialgeometrie der {K}reise und Kugeln},
bearbeitet von G. Thomsen, Grundlehren der mathematischen
Wissenschaften, 29, Springer, Berlin, 1929.

\bibitem{Bobenko2006}
A. I. Bobenko, T. Hoffmann, and B. A. Springborn,
\emph{Minimal surfaces from circle patterns: geometry from  combinatorics},
{Ann. of Math. (2)} {164} (2006), no. 1, 231--264.

\bibitem{Bobenko2007}
A. I. Bobenko and Y. Suris,
\emph{On discretization principles for differential geometry. The geometry of spheres},
{Russian Math. Surveys} {62} (2007), no. 1, 1--43.

\bibitem{Bobenko2010}
A. I. Bobenko, H. Pottmann, and J. Wallner,
\emph{A curvature theory for discrete surfaces based on mesh parallelity},
{Math. Ann.} {348} (2010), no. 1, 1--24.

\bibitem{Bo-Pe2009}
C. Bohle and G. P. Peters,
\emph{Bryant surfaces with smooth ends},
{Comm. Anal. Geom.} {17} (2009), no. 4, 587--619;
arXiv:math/0411480v2.

\bibitem{Bo2012}
C. Bohle,
\emph{Constant mean curvature tori as stationary solutions
to the Davey--Stewartson equation},
{Math. Z.} {217} (2012), 489--498.

\bibitem{BCG}
{R. L. Bryant, S.-S. Chern, and P. A. Griffiths},
{\em Exterior differential systems},
Proceedings of the 1980 Beijing Symposium on Differential Geometry and Differential
Equations, Vol. 1, Sci. Press Beijing, Beijing, 1982, 219--338.

\bibitem{Br-duality}
R. L. Bryant,
\emph{A duality theorem for Willmore surfaces},
{J. Differential Geom.} {20} (1984), 23--53.

\bibitem{Br1987}
R. L. Bryant,
\emph{Surfaces of mean curvature one in hyperbolic space},
Th\'eorie des vari\'et\'es minimales et applications (Palaiseau,
1983--1984), {\em Ast\'erisque} {154-155} (1987), 321--347.

\bibitem{BCGGG}
{R. L. Bryant, S.-S. Chern, R. B. Gardner, H. L. Goldschmidt, and P. A. Griffiths},
{\em Exterior differential systems}, MSRI Publications, 18, Springer-Verlag,
New York, 1991.

\bibitem{Ca1903}
P. Calapso,
\emph{Sulle superficie a linee di curvatura isoterme},
{Rendiconti Circolo Matematico di Palermo} {17} (1903), 275--286.

\bibitem{Ca1915}
P. Calapso,
\emph{Sulle trasformazioni delle superficie isoterme},
{Ann. Mat. Pura Appl.} {24} (1915), 11--48.


\bibitem{Ca1}
E. Cartan,
\emph{\em Sur le probl\`eme g\'en\'eral de la
d\'eformation}, C. R. Congr\'es Strasbourg (1920), 397--406; or
{Oeuvres Compl\`{e}tes}, III 1, 539--548.

\bibitem{Cartan-book}
{E. Cartan},
{\em Les syst\`emes diff\'erentiels ext\'erieurs et leurs applications g\'eom\'etriques},
Hermann, Paris, 1945.

\bibitem{Ce}
T. E. Cecil,
{\em Lie Sphere Geometry. With applications to submanifolds},
Springer-Verlag, New York, 1992.

\bibitem{CGP-BAMS}
P. T. Chru\'sciel, G. J. Galloway, and D. Pollack,
\emph{Mathematical general relativity: a sampler},
{Bull. Amer. Math. Soc. (N.S.)} {47} (2010), 567--638.

\bibitem{Gr}
P. A. Griffiths,
\emph{On Cartan's method of Lie groups and moving frames as applied
to uniqueness and existence questions in differential geometry},
{Duke Math. J.} {41} (1974), 775--814.

\bibitem{Grbook}
{P. A. Griffiths},
\textit{Exterior Differential Systems and the Calculus of
Variations}, Progress in Mathematics, 25, Birkh\"auser, Boston, 1982.

\bibitem{GJbook}
{P. A. Griffiths and G. R. Jensen},
 {\em Differential Systems and Isometric Embeddings},
 Annals of Mathematics Studies, 114, Princeton University Press, Princeton, NJ, 1987.

\bibitem{HE}
S. W. Hawking and G. F. R. Ellis,
{\em The large scale structure of space-time}.
Cambridge Monographs on Mathematical Physics, 1,
Cambridge University Press, London--New York, 1973.

\bibitem{HMN}
U. Hertrich-Jeromin, E. Musso, and L. Nicolodi,
\emph{M\"obius geometry of surfaces of constant mean curvature 1 in hyperbolic space},
{\em Ann. Global Anal. Geom.} {19} (2001), 185--205.

\bibitem{HJlibro}
U. Hertrich-Jeromin,
{\em Introduction to M\"obius Differential Geometry},
London Mathematical Society Lecture Note Series, 300,
Cambridge University Press, Cambridge, 2003.

\bibitem{ILlibro}
 {T. A. Ivey and J. M. Landsberg},
{\em Cartan for Beginners: Differential Geometry via Moving Frames and
Exterior Differential Systems},
Graduate Studies in Mathematics, 61,
American Mathematical Society, Providence, RI, 2003.

\bibitem{J}
G. R. Jensen,
\emph{Deformation of submanifolds of homogeneous spaces},
{J. Differential Geom.} {16} (1981), 213--246.

\bibitem{Gary}
G. R. Jensen,
\emph{Lie sphere geometry},
Proceedings of the {\em Workshop on Geometry of Lagrangian Grassmannians and nonlinear PDEs},
IMPAN, Warsaw, 5-9 September 2016.

\bibitem{JMN-JPA}
{G. R. Jensen, E. Musso, and L. Nicolodi},
\emph{The geometric Cauchy problem for the membrane shape equation},
{J. Phys. A} {47} (2014), no. 49, 495201, 22 pp.

\bibitem{JMN}
{G. R. Jensen, E. Musso, and L. Nicolodi},
{\em Surfaces in Classical Geometries. A Treatment by Moving Frames},
Universitext, Springer, Cham, 2016.

\bibitem{Kob}
O. Kobayashi, 
\emph{Maximal surfaces in the 3-dimensional Minkowski space $L^{3}$},
{Tokyo J. Math.} {6} (1983), no. 2, 297--309.

\bibitem{Lawson}
H. B. Lawson,
\emph{Complete minimal surfaces in $S^{3}$},
{Ann. of Math. (2)} {92} (1970), 335--374.

\bibitem{Lee2005}
S. Lee,
\emph{Spacelike surfaces of constant mean curvature $\pm1$
in de Sitter  3-space ${\mathbb S}^3_1(1)$},
{Illinois J. Math.} {49} (2005), no. 1, 63--98.

\bibitem{li-li-wang}
T. Li, H. Li, and C. Wang,
\emph{Classification of hypersurfaces with parallel Laguerre second fundamental form in $\mathbb R^n$},
{Differential Geom. Appl.} {28} (2010), no. 2, 148--157.

\bibitem{li-wang-mm}
T. Li and C. P. Wang,
\emph{Laguerre geometry of hypersurfaces in $\mathbb{R}^{n}$},
{Manuscripta math.} {122} (2007), 73--95.

\bibitem{MaMuNi}
J. M. Manzano, E. Musso, and L. Nicolodi,
\emph{Bj\"orling type problems for elastic surfaces},
{Rend. Semin. Mat. Univ. Politec. Torino} {74} (2016), no. 1, 211--231.

\bibitem{McKay}
B. McKay,
\emph{Introduction to exterior differential systems},
Proceedings of the {\em Workshop on Geometry of Lagrangian Grassmannians and nonlinear PDEs},
IMPAN, Warsaw, 5-9 September 2016.

\bibitem{Mu-TS}
E. Musso,
\emph{Deformation of surfaces in M\"obius space},
Rend. Istit. Mat. Univ. Trieste {27} (1995),  no. 1-2, 25--45.

\bibitem{MN-REND-RM}
E. Musso and L. Nicolodi,
\emph{$L$-minimal canal surfaces},
{Rend. Matematica} {15} (1995), 421--445.

\bibitem{MN-TAMS}
E. Musso and L. Nicolodi,
\emph{A variational problem for surfaces in Laguerre geometry},
{Trans. Amer. Math. Soc.} {348} (1996), 4321--4337.

\bibitem{MN-BOLL}
E. Musso and L. Nicolodi,
\emph{Isothermal surfaces in Laguerre geometry},
{Boll. Un. Mat. Ital. (7) II-B}, Suppl. fasc. 2, {1997}, 125--144.

\bibitem{MN-BUD}
E. Musso and L. Nicolodi,
{\em On the equation defining isothermic surfaces in Laguerre geometry},
New Developments in Differential Geometry, Budapest 1996,
Kluver Academic Publishers, Dordrecht, The Netherlands, 285--294.

\bibitem{MN-AMB}
E. Musso and L. Nicolodi,
\emph{Laguerre geometry of surfaces with plane lines of curvature},
{Abh. Math. Sem. Univ. Hamburg} {69} (1999), 123--138.

\bibitem{MN-IJM}
E. Musso and L. Nicolodi,
\emph{The Bianchi-Darboux transform of $L$-isothermic surfaces},
{Internat. J. Math.} {11} (2000), no. 7, 911--924.

\bibitem{MNjmp}
{E. Musso and L. Nicolodi},
\emph{On the Cauchy problem for the integrable system of Lie minimal surfaces},
{J. Math. Phys.} {46} (2005), no. 11, 3509--3523.

\bibitem{MN-TOH}
E. Musso and L. Nicolodi,
\emph{Deformation and applicability of surfaces in Lie sphere geometry},
{Tohoku Math. J. (2)} {58} (2006), no. 2, 161--187.

\bibitem{MNphysD}
{E. Musso and L. Nicolodi},
\emph{A class of overdetermined systems defined by tableaux: involutiveness and the Cauchy problem},
{Phys. D} {229} (2007), no. 1, 35--42.

\bibitem{MN-HOUSTON}
E. Musso and L. Nicolodi,
\emph{Conformal deformation of spacelike surfaces in Minkowski space},
{Houston J. Math.} {35} (2009), no. 4, 1029--1049.

\bibitem{MNjgp}
{E. Musso and L. Nicolodi},
Marginally outer trapped surfaces in de Sitter space by low-dimensional geometries,
{\em J. Geom. Phys.} \textbf{96} (2015), 168--186.

\bibitem{MN-MZ}
E. Musso and L. Nicolodi,
\emph{Holomorphic differentials and Laguerre deformation of surfaces},
{Math Z.} {284} (2016), no. 3-4, 1089--1110.

\bibitem{Pa1999}
B. Palmer,
\emph{Remarks on a variational problem in Laguerre geometry},
{Rend. Mat. Appl. (7)} {19} (1999), no. 2, 281--293.

\bibitem{Pa1990}
B. Palmer,
\emph{Spacelike constant mean curvature surfaces in pseudo-Riemannian space forms},
{Ann. Global Anal. Geom.} {8} (1990), 217--226.

\bibitem{Pa2015}
B. Palmer,
\emph{Anisotropic wavefronts and Laguerre geometry},
J. Math. Phys. 56 (2015), no. 2, 023503, 10 pp.

\bibitem{Pember}
M. Pember,
\emph{Lie applicable surfaces}, arXiv:1606.07205 [math.DG].

\bibitem{Pott1998}
H. Pottmann and M. Peternell,
\emph{Applications of Laguerre geometry in CAGD},
{Comput. Aided Geom. Design} {15} (1998), no. 2, 165--186.

\bibitem{Pott2009}
H. Pottmann, P. Grohs, and N. J. Mitra,
\emph{Laguerre minimal surfaces, isotropic geometry and linear elasticity},
{Adv. Comput. Math.} {31} (2009), no. 4, 391--419.

\bibitem{Szer}
A. Szereszewski,
\emph{$L$-isothermic and $L$-minimal surfaces},
{J. Phys. A} {42} (2009), no. 11, 115203--115217.

\bibitem{Pott2012}
M. Skopenkov, H. Pottmann, and P. Grohs,
\emph{Ruled Laguerre minimal surfaces},
{Math. Z.} {272} (2012), no. 1-2, 645--674.

\bibitem{Rog-Szer}
C. Rogers and A. Szereszewski,
\emph{A B\"acklund transformation for $L$-isothermic surfaces},
{J. Phys. A} {42}  (2009),  no. 40, 404015, 12 pp.

\bibitem{Song2013}
Y.-P. Song,
\emph{Laguerre isothermic surfaces in ${\mathbb R}^3$ and their Darboux transformation},
{Sci. China Math.} {56} (2013), no. 1, 67--78.

\bibitem{Thomsen}
G. Thomsen,
\emph{\"Uber konforme Geometrie I: Grund\-lagen der konformen Fl\"a\-chen\-theorie},
{Hamb. Math. Abh.} {3} (1923), 31--56.

\bibitem{UY-Crelle}
M. Umehara and K. Yamada,
\emph{A parametrization of the Weierstrass formulae
and perturbation of complete minimal surfaces in $\R^3$
into the hyperbolic 3-space},
{J. Reine Angew. Math.} {432} (1992), 93--116.

\bibitem{Voss1985}
K. Voss,
\emph{Verallgemeinerte Willmore-Fl\"achen},
Mathematisches Forschungsinstitut Oberwolfach, Workshop Report 42 (1985), 22--23.


\end{thebibliography}

\end{document}